\documentclass{amsart}
\usepackage[utf8]{inputenc}
\usepackage{amsmath}
\usepackage{amssymb}
\usepackage{amsfonts}
\usepackage{graphicx}
\usepackage{bbm}
\usepackage{amsthm}
\usepackage{enumitem}
\usepackage{mathtools}
\usepackage{dsfont}
\usepackage{esint}
\usepackage[utf8]{inputenc}

\newcommand{\odd}{\text{ odd}}
\newcommand{\even}{\text{ even}}
\newcommand{\mainerror}{O_{\underline{n},\underline{m},k}\left( \frac{(\log q)^{2n}}{\sqrt{q}} \right)}

\newcommand{\prob}{\textbf{Prob}}
\newcommand*{\ndiv}{\not|}

\newtheorem{theorem}{Theorem}[section]

\newtheorem{lemma}[theorem]{Lemma}
\newtheorem{proposition}[theorem]{Proposition}

\theoremstyle{definition}
\newtheorem{example}[theorem]{Example}
\newtheorem{definition}{Definition}[section]
\theoremstyle{remark}
\newtheorem{remark}{Remark}[section]

\theoremstyle{plain} % just in case the style had changed
\newcommand{\thistheoremname}{}
\newtheorem*{nono-theorem}{\thistheoremname}

\DeclareRobustCommand{\prob}[1][P]{\ensuremath {\mathbb{#1}}}

\newcommand{\Addresses}{{% additional braces for segregating \footnotesize
  \bigskip
  \footnotesize

  A.~Hussain, \textsc{Department of Mathematics, University of Bristol,
    United Kingdom}\par\nopagebreak
  \textit{E-mail address}, \texttt{ayesha.hussain@bristol.ac.uk}
        
        }}

\begin{document}

\title{The Limiting Distribution of Character Sums}
\author{Ayesha Hussain}
\date{6th July 2021}

\begin{abstract}
        In this paper, we consider the distribution of the continuous paths of Dirichlet character sums modulo prime $q$ on the complex plane. We also find a limiting distribution as $q \rightarrow \infty$ using Steinhaus random multiplicative functions, stating properties of this random process. This is motivated by Kowalski and Sawin's work on Kloosterman paths. 
\end{abstract}

\maketitle

\section{Introduction}

     Given a primitive Dirichlet character $\chi$ modulo $q$, we define the normalised partial character sum
     \begin{align*}
        S_\chi(t) \coloneqq \frac{1}{\sqrt{q}} \sum_{n \leq qt} \chi(n),
    \end{align*}
    for $t \in [0,1]$. Such character sums play a fundamental role in analytic number theory. Our goal is to study the distribution of character sums for prime modulus $q$, and find the limiting distribution as $q \rightarrow \infty$, answering the open problem set by Kowalski and Sawin \cite[Section 5.2]{kowalskisawin2016}.
    
    When investigating the maximum of these character sums, Bober, Goldmakher, Granville and Koukoulopoulos \cite{bggk2014} studied the distribution function for $\tau > 0$,
     \begin{align}\label{eq: phi probability from 4 author paper}
         \Phi_q(\tau) \coloneqq \frac{1}{\phi(q)} \# \bigg \{ \chi \mod q : \max_t |S_\chi(t)| > \frac{e^\gamma}{\pi}\tau \bigg \}.
     \end{align}
     The limiting distribution of $\Phi_q$ is
     \begin{align*}
         \Phi(\tau) \coloneqq \mathbb{P}\left(\max_t |F(t)| > 2e^\gamma \tau \right),
     \end{align*}
     where $F(t)$ is a random Fourier series defined later in this paper \cite[Theorem 1.4]{bggk2014}. We find, through different methods, that the limiting distribution of character sums, not just their maxima, uses the same random series. Our main theorem can also be used to recover their result.
     
    The character sum $S_\chi(t)$ is a step function, with jump discontinuities at every $t \in \frac{1}{q}\mathbb{Z}$. In order to circumvent the difficulties posed by these discontinuities, it is natural to consider a continuous modification, where the steps are replaced by straight line interpolations. 
    \begin{definition}
    Character paths are paths in the complex plane formed by drawing a straight line between the successive partial sums
    \begin{align*}
        \begin{array}{cc}
             S_\chi(x) = \frac{1}{\sqrt{q}} \sum_{n \leq qx} \chi(n),
             &  S_\chi(x + 1/q) = \frac{1}{\sqrt{q}} \sum_{n \leq qx + 1} \chi(n),
        \end{array}
    \end{align*}
    for $x \in [0,1)$ and $qx \in \mathbb{Z}$. 
     We parameterise character paths by the function
    \begin{align*}
        f_\chi(t) \coloneqq S_\chi(t) + \frac{\{qt\}}{\sqrt{q}} \chi\left(\lceil qt\rceil \right),
    \end{align*}
    where $\{x\}$ is the fractional part of the number $x$. 
    \end{definition}
    Note the difference between character sums and character paths is bounded by $\frac{1}{\sqrt{q}}$. Character paths, like character sums, are periodic and have the truncated Fourier series
    \begin{align}\label{eq: fourier series of character sum}
        \frac{\tau(\chi)}{2 \pi i \sqrt{q}} \sum_{0 < |k| < q} \frac{\overline{\chi}(k)}{k} (1 - e(-kt)) + O \left( \frac{\log q}{\sqrt{q}}\right).
    \end{align}
    Character paths are polygonal, continuous and closed. Examples of character paths can be seen in Figure \ref{fig: mod10007}.
    
    \begin{figure} 
        \centering
        \begin{tabular}{cc}
            \includegraphics[width=7cm]{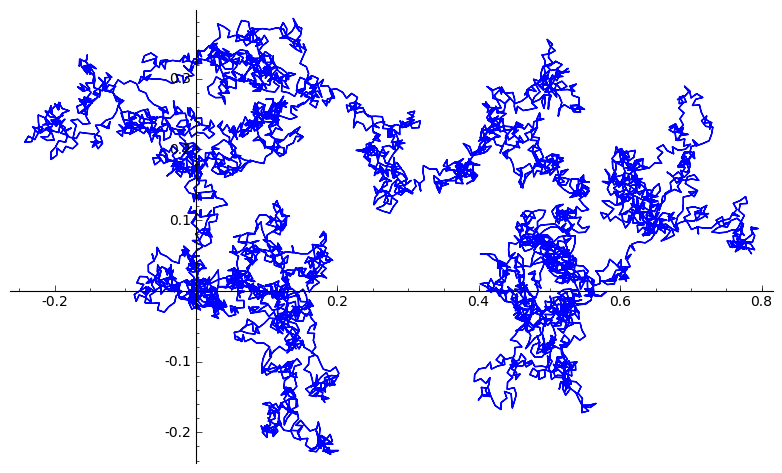} &\includegraphics[width=5cm]{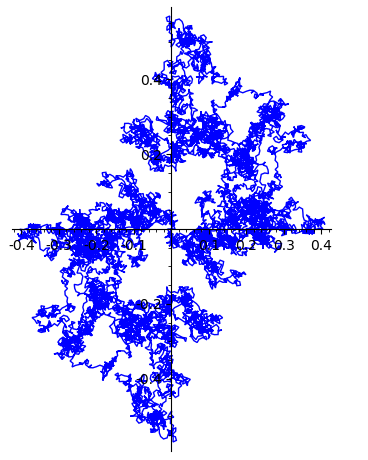} \\
            \small{Odd path defined by $\chi(5) = e\left( \frac{1}{10007} \right)$.} &
            \small{Even path defined by $\chi(5) = e\left( \frac{2}{10007} \right)$.}
        \end{tabular}
        \caption{Character paths of an odd and even character modulo $10007$.}
        \label{fig: mod10007}
    \end{figure}
    
     We define the distribution of character paths by taking $\chi \text{ mod } q \mapsto f_\chi(t)$ as a random process, choosing $\chi$ uniformly at random. Let us write $\mathcal{F}_q$ for this random process,
     \begin{align*}
            \mathcal{F}_{q}(t)\coloneqq \{ f_\chi(t) : \chi \text{ mod } q\}.
    \end{align*}
    One of the main goals of this paper is to find the limiting distribution of the sequence $(\mathcal{F}_{q})_{q}$ as $q \rightarrow \infty$ through the primes. For this, we need Steinhaus random multiplicative functions.
    \begin{definition}\label{def: Steinhaus rmf} 
        \begin{enumerate}
            \item  Steinhaus random variables $X_p$ are random variables uniformly distributed on the unit circle $\{|z| = 1\}$. 
            \item Steinhaus random multiplicative functions $X_n$, $n \in \mathbb{N}$, are defined as
        \begin{align*}
            X_n = \prod_{p^a \| n} X_p^a,
        \end{align*}
        where $X_p$ are Steinhaus random variables for prime $p$. We extend this definition to $n \in \mathbb{Z}$ by taking $X_{-1} = \pm 1$ with probability $1/2$ each, so $X_{-n} = X_{-1}X_n$.
        (Here $p^a \| n$ means $p^a$ strictly divides $n$, so $p^a|n$ but $p^{a+1}\not| n$).
        \end{enumerate}
    \end{definition}
    
       Steinhaus random multiplicative functions are completely multiplicative, with all values distributed on the unit circle. This leads to a natural question: can we compare partial sums of characters with partial sums of Steinhaus random multiplicative functions, assuming similar behaviour?
       Sums of Steinhaus random multiplicative functions might be a good model for short character sums, but the periodicity of the characters means that for long character sums the model fails. Instead, we consider $X_n$ as Fourier coefficients. The sums of Steinhaus random multiplicative functions have a long history. See Harper \cite{harperRMFI2017} for an example of recent work on this.
    
    Consequently, we must find a way to incorporate the periodicity from the character sums.
    Let $F(t)$ be the random Fourier series
    \begin{align} \label{eq: random F plus and minus}
            F(t) \coloneqq \frac{\eta}{2\pi} \sum_{|k| > 0} \frac{X_k}{k}( 1 - e(kt)), 
    \end{align}
    where $X_k$ are Steinhaus random multiplicative functions for $k \neq 0$ and $\eta$ is a random variable uniformly distributed on the unit circle. Additionally, take $F_{\pm}(t)$ where we fix $X_{-1}$ as $+1$ or $-1$. The infinite series is well defined, and we show in Section \ref{s: Properties of F} that this is almost surely the Fourier series of a continuous function. Therefore we can think of $F$ as a random process on $C([0,1])$. Examples of the paths formed by $F_{\pm}$ are shown in Figure \ref{fig: random10007}.
        
    \begin{figure} 
        \centering
        \begin{tabular}{cc}
            \includegraphics[scale=.35]{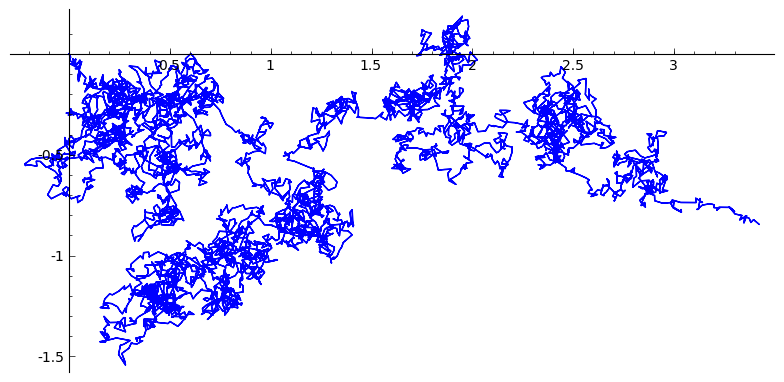} & \includegraphics[scale=.3]{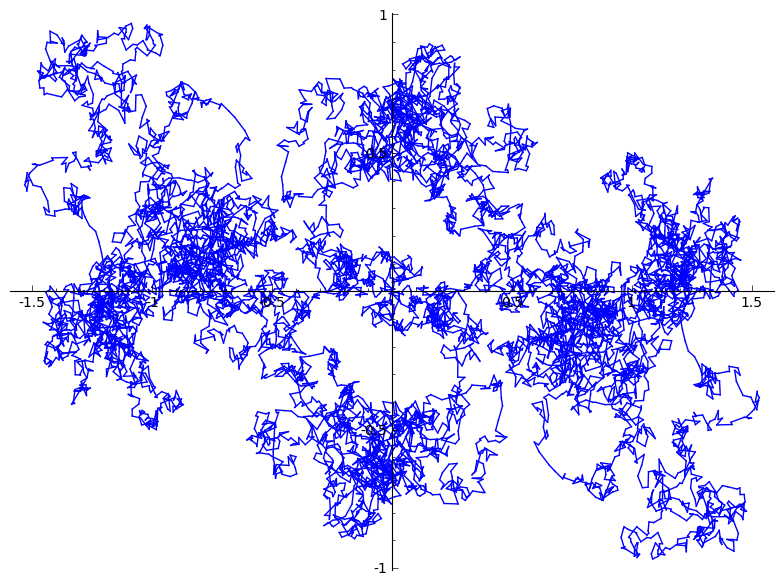}\\
            \small{A sample of $F_-(t)$.} &
            \small{A sample of $F_+(t)$.}
        \end{tabular}
        \caption{Samples of $F_{\pm}$ with $10007$ terms.}
        \label{fig: random10007}
    \end{figure}
    
    Using the random Fourier series $F$, we state the main theorem of the paper:
    \begin{theorem}\label{thm: main theorem}
        Let $F_\pm$ be defined as above for $t \in [0,1]$. The sequence of the distributions of character paths $(\mathcal{F}_{q,\pm}(t))_{q}$ weakly converges to the process $F_{\pm}(t)$ in the Banach space $C([0,1])$ as $q \rightarrow \infty$ through the primes.
        In other words, for any continuous and bounded map
        \begin{align*}
            \psi: C([0,1]) \rightarrow \mathbb{C},
        \end{align*}
        we have, for prime $q$,
        \begin{align*}
            \lim_{\substack{q \rightarrow \infty}} \mathbb{E}\left(\psi\left(\mathcal{F}_{q,\pm}\right)\right) = \mathbb{E}\left( \psi(F_\pm) \right).
        \end{align*}
    \end{theorem}
   
\subsection{Proof Outline}
    Theorem \ref{thm: main theorem} shows that Steinhaus random multiplicative functions can be used as Fourier coefficients of a random process $F$ to find the limiting distribution of character paths. In Section \ref{s: Properties of F} we properly define the random process $F$ and prove some of its properties. Theorem \ref{thm: main theorem} only makes sense if $F(t)$ is a function almost surely in $C([0,1])$, which we prove in Theorem \ref{thm: F almost surely continuous}. 
    
    The proof of Theorem \ref{thm: main theorem} can be split into two parts: proving that the sequence $(\mathcal{F}_{q,\pm}(t))_{q}$ converges in finite distributions to the random process $F_\pm(t)$ and that the sequence of distributions is relatively compact. Convergence in finite distributions is proved in Section \ref{s: proof of main thm}, using the method of moments. To prove relative compactness, it is sufficient to use Prohorov's Theorem \cite[Theorem 5.1]{billingsley2013}, which states that if a family of probability measures is tight, then it must be relatively compact. Section \ref{s: tightness} proves the sequence of distributions $(\mathcal{F}_{q,\pm}(t))_{q}$ satisfies the tightness property, therefore proving Theorem \ref{thm: main theorem}.

   \begin{remark}
        As referred to earlier, Bober, Goldmakher, Granville and Koukoulopoulos \cite{bggk2014} investigated the distribution function $\Phi_q(\tau)$, defined in Equation \eqref{eq: phi probability from 4 author paper}. They proved $(\Phi_{q,\pm}(\tau))$ converges weakly to 
        \begin{align*}
            \textbf{Prob}\left( \max_{t}|F_\pm(t)| > 2e^\gamma \tau \right).
        \end{align*}
        as $q \rightarrow \infty$ through the primes. Theorem \ref{thm: main theorem} can be used to obtain the same result.
    \end{remark}
     \begin{remark}
    The definition of character paths is motivated by similar research by Kowalski and Sawin \cite{kowalskisawin2016, kowalskisawin2017}. In their papers they define \textit{Kloosterman paths}, $K_p(t)$, view the paths as random variables, and find their limiting distribution as $p \rightarrow \infty$. My work continues in this vein to investigate the analogous limiting distribution of character paths. Due to the multiplicativity of Dirichlet characters, our random multiplicative coefficients $X_n$ aren't independent. This leads to interesting properties shown in Section \ref{s: Properties of F}.
    \end{remark}
    
    \begin{remark}
        Theorem \ref{thm: main theorem} is restricted to $q$ being prime, so a natural question is to consider composite $q$ as well. Steinhaus random multiplicative functions are non zero so we need a high percentage of primitive characters modulo $q$. If we take $q$ not being too smooth we might be able to relax this condition, as the contribution from imprimitive characters could potentially be included in the error terms already produced from the method. Future work could explore the generalised case when the modulus of the characters is not prime.
    \end{remark}
   
\subsection*{Notation}
    We follow the usual convention of $e(x) = e^{2\pi i x}$ for all $x \in \mathbb{R}$ and $p^k \| n$ to be when $p^j|n$ for all $1 \leq j \leq k$ and $p^{k+1} \ndiv n$. We take $d_N(x)$ as the $N$th divisor function for $x \in \mathbb{N}$,
    \begin{align*}
        d_N(x) = \sum_{x_1 \cdots x_N = x} 1.
    \end{align*}
    Given a positive integer $n$, we define $P^+(n)$ and $P^-(n)$ as the largest and smallest prime divisors of $n$. We take $P^+(1) = 1$ and $P^-(1) = \infty$ as in conventional notation.
    
\subsection*{Acknowledgements}
    I would like to thank my Ph.D. supervisor Jonathan Bober for his guidance, as well as Andrew Granville and the anonymous referee for their comments.

\section{Properties of $F(t)$}\label{s: Properties of F}
    Recall the random process $F$, defined by the infinite sum
    \begin{align*}
        F(t) = \frac{\eta}{\pi} \sum_{n \neq 0} \frac{1 - e(nt)}{n}X_n,
    \end{align*}
    where $X_n$ are Steinhaus random multiplicative functions, defined in Definition \ref{def: Steinhaus rmf}, and $\eta$ is a random variable uniformly distributed on the unit circle. We define the infinite sum as a limit of the smooth sum
    \begin{align*}
        \frac{\eta}{\pi} \sum_{\substack{n \neq 0\\ P^+(|n|) \leq y}} \frac{1 - e(nt)}{n}X_n,
    \end{align*}
    as $y \rightarrow \infty$. In this section we deal with samples of the random process and prove some of their properties. We will be using sample functions of the random process, which we define here for ease of notation.
    
    \begin{definition}
    We say $G$ is a sample function of the random process $F$, where
    \begin{align*}
        G(t) \coloneqq \frac{c}{\pi} \sum_{n \neq 0} \frac{1 - e(nt)}{n}\alpha_n,
    \end{align*}
    for $c$ a sample of a Steinhaus random variable and $\alpha_n$ is a sample of a Steinhaus random multiplicative function.
    \end{definition}
    
    The Fourier coefficients are bounded by $O(1/n)$. This will be useful later in the section, where we show the infinite series $F$ can also be defined as the limit of partial symmetric sums.

    Our main result of Section \ref{s: Properties of F} is proving any sample function of $F$ is almost surely continuous. This is non trivial and involves considering the $y$-smooth and ``$y$-rough'' parts of the infinite sum $G(t)$.
    
    Let
    \begin{align*}
        &S_y \coloneqq \sum_{\substack{n \neq 0 \\ P^+(|n|) \leq y}} \frac{1 - e(nt)}{n} \alpha_n
        &\text{ and } &
        &R_y \coloneqq \sum_{\substack{n \neq 0 \\ P^+(|n|) > y}} \frac{1 - e(nt)}{n} \alpha_n,
    \end{align*}
    where $P^+(n)$ is the largest prime factor of $n$. Note that $S_y + R_y = \frac{\pi}{c}G(t)$ and the two functions are not independent. %These sums can also be seen as smooth and rough samples of the random process $F$.
    
    \begin{lemma}\label{thm: rough part is small}
        For all $\epsilon > 0$ and sufficiently large $y > 1$,
        \begin{align*}
            \prob \left( \|R_y\|_\infty > \epsilon \right) \ll \exp\left\{ \frac{-\epsilon^2 y^{1/3}}{\log y ( \log y + O(1))}\left( \log \left(\frac{(\log y)^{20}}{\log y + O(1)}\right) +  O\left(\log \epsilon\right)\right)\right\}
        \end{align*}
        independently of $S_y$, where $\| \cdot \|_\infty \coloneqq \max_{t \in [0,1]} |\cdot |$.  
        Notably for all $\epsilon > 0$,  we have $ \prob \left( \|R_y\|_\infty > \epsilon \right) \rightarrow 0$ as $y \rightarrow \infty$.
    \end{lemma}
    
     \begin{proof}
        For all $y \geq 1$,
        \begin{align*}
            \sum_{\substack{n \geq 1 \\ P^+(n) > y} } \frac{1 - e(nt)}{n}\alpha_n 
            = \sum_{\substack{n \geq 1 \\ P^+(n) \leq y} } \frac{\alpha_n}{n}\sum_{\substack{m > y \\ P^-(m) > y}} \frac{1 - e(mnt)}{m} \alpha_m, 
        \end{align*}
        where $P^-(m)$ is the smallest prime factor of $m$. By setting
        \begin{align*}
            T(\alpha) \coloneqq \max_{t \in [0,1]} \left| \sum_{\substack{m > y \\ P^-(m) > y}} \frac{1 - e(mt)}{m} \alpha_m \right|,
        \end{align*}
        we have
        \begin{align*}
            \| R_y\|_\infty \coloneqq \max_{t \in [0,1]} \left| \sum_{\substack{n \neq 0 \\ P^+(n) > y}} \frac{1 - e(nt)}{n} \alpha_n \right| \leq 2 \sum_{\substack{n \geq 1 \\ P^+(n) \leq y}} \frac{T(\alpha)}{n}.
        \end{align*}
        We then use the following result \cite[Theorem 2.7]{montgomeryvaughan2006},
        \begin{align*}
            \sum_{P+(n) \leq y} \frac{1}{n} = e^\gamma \log y + O(1),
        \end{align*}
        where we took the limit of the Lemma.
        Consequently, $\|R_y\|_\infty$ is bounded above by
        \begin{align*}
            \|R_y\|_\infty \leq 2 T(\alpha) \left(e^\gamma \log y + O(1)\right).
        \end{align*}
        
        Adapting \cite[Proposition 5.2]{bggk2014} for Steinhaus random multiplicative functions\footnote{We take $q\rightarrow \infty$ and apply a sample of a Steinhaus random multiplicative function $\alpha_n$ instead of $\chi(n)$. }, for $k \geq 3$ and $y \geq k^3$,
        \begin{align*}
            \mathbb{E} \left[ \left( \sum_{\substack{n > y \\ P^-(n) > y}} \frac{1 - e(mt)}{m} \alpha_m \right)^{2k} \right] \ll \frac{1}{(\log y)^{40k}}.
        \end{align*}
        We set $\rho_y$ as the probability $T(\alpha) > \epsilon(y)$ and 
        \begin{align*}
            k = \left\lfloor \frac{\epsilon(y)^2 y^{1/3}}{\log y} \right\rfloor.
        \end{align*}
        Therefore, 
        \begin{align*}
            \rho_y \leq  \frac{\mathbb{E}(T(\alpha)^{2k})}{\epsilon(y)^{2k}}  
            \ll \frac{\epsilon(y)^{-2k}}{(\log y)^{40k}} 
            \leq \left( \frac{\epsilon(y)^{-1}}{(\log y)^{20}} \right)^{\frac{2\epsilon(y)^2 y^{1/3}}{\log y}}.
        \end{align*}
        Taking $\epsilon(y) = \frac{\epsilon}{2e^\gamma \log y + O(1)}$ for $\epsilon > 0$,
        \begin{align*}
            \mathbb{P} \left( \|R_y\|_\infty > \epsilon \right) 
            &\leq \mathbb{P} \left( T(\alpha) > \frac{\epsilon}{2e^\gamma \log y + O(1)}\right) 
            \ll \left( \frac{2e^\gamma \log y + O(1)}{\epsilon(\log y)^{20}} \right)^{\frac{2\epsilon^2 y^{1/3}}{\log y(2e^\gamma \log y + O(1))}} \\
            &\ll \exp\left( \frac{-\epsilon^2 y^{1/3}}{\log y ( \log y + O(1))}\left( \log \left(\frac{(\log y)^{20}}{\log y + O(1)}\right) +  \log\left( O(\epsilon)\right)\right)\right).
        \end{align*}
        
        To prove the final part of the lemma, we take $y\rightarrow \infty$ to show the probability tends to $0$,
        \begin{align*}
           0 \leq \mathbb{P}\left( \| R_y \|_\infty > \epsilon \right) \ll \lim_{y \rightarrow \infty} \exp\left( \frac{-\epsilon^2 y^{1/3}}{\log y ( \log y + O(1))}\left( \log \left(\frac{(\log y)^{20}}{\log y + O(1)}\right) +  O\left( \log\epsilon\right)\right)\right) = 0.
        \end{align*}
    \end{proof}
    
    Lemma \ref{thm: rough part is small} can be appreciated more by taking $\epsilon = 1/\log y$, leading to the following example.
    \begin{example}
        For sufficiently large $y > 1$,
        \begin{align*}
            \prob \left( \|R_y\|_\infty > \frac{1}{\log y} \right) \ll \exp \left\{\frac{- y^{1/3}}{(\log y)^3 ( \log y + O(1))}\left( \log \left(\frac{(\log y)^{20}}{\log y + O(1)}\right) + O  \left(\log \log y\right)\right)\right\}
        \end{align*}
        independently of $S_y$.
    \end{example}
    
    Subsequently, defining $F$ as the limit of its smooth parts, we get the following theorem.
    \begin{theorem}\label{thm: F almost surely continuous}
        Let $F$ be the random Fourier series
        \begin{align*}
            F(t) \coloneqq \lim_{y \rightarrow \infty} \frac{\eta}{\pi} \sum_{\substack{n \neq 0 \\ P^+(|n|) \leq y}} \frac{1 - e(nt)}{n}X_n,
        \end{align*} 
        where $X_n$ are Steinhaus random multiplicative functions and $\eta$ is a random variable uniformly distributed on $\{ |z| = 1\}$.
        With probability 1 this is the Fourier series of a continuous function.
    \end{theorem}
    \begin{proof}
        We will prove this theorem by showing any sample of the random process $F$ is almost surely continuous.
        
        Consider the sequence of functions $(S_y)_y$ and $(R_y)_y$ defined by
        \begin{align*}
            S_y(t) \coloneqq \frac{c}{\pi} \sum_{\substack{n \neq 0 \\ P^+(|n|) \leq y}} \frac{1 - e(nt)}{n} \alpha_n, \\
            R_y(t) \coloneqq \frac{c}{\pi} \sum_{\substack{n \neq 0 \\ P^+(|n|) > y}} \frac{1 - e(nt)}{n} \alpha_n,
        \end{align*}
        where $c,\alpha_n$ are on the unit circle and $\{\alpha_n\}$ are completely multiplicative. Note that samples of the random process $F$ can be written as $S_y(t) + R_y(t)$ for appropriate choices of $\alpha_n$ and $c$.
        
        The function $S_y$ is the $y-$smooth part of a sample of the random process $F$, which we call $G(t)$. The sum $S_y$ converges absolutely, so $S_y(t)$ is a continuous function.
        
        To prove continuity with probability 1, we use the first Borel-Cantelli Lemma \cite[Chapter 2, Theorem 18.1]{gut2013} to show the $y-$rough part of $G$ vanishes as $y \rightarrow \infty$. Consider the sequence $\{ R_y : \|R_y\|_\infty > \epsilon\}_y$, where $\|R_y\|_\infty$ is the maximum of the $y-$rough part of $G$. Using Example \ref{thm: rough part is small},
        \begin{align*}
            \sum_{y = 1}^\infty \prob (\| R_y\|_\infty > \epsilon) \ll \sum_{y = 1}^\infty \exp \left\{\frac{- y^{1/3}}{(\log y)^3 ( \log y + O(1))}\left( \log \left(\frac{(\log y)^{20}}{\log y + O(1)}\right) + O  \left(\log \log y\right)\right)\right\} < \infty.
        \end{align*}
        As a result, the probability of $\|R_y\|_\infty > \epsilon$ occuring infinitely often is zero. Therefore, almost surely the rough part of the sample of $F$ is less than $\epsilon$. Consequently, $R_y$ vanishes almost surely as $y \rightarrow \infty$.
        
        As a result, the sequence of continuous functions $(S_y)$ uniformly converges to its limit, which by the Uniform Limit Theorem must be continuous. By uniform convergence we can compute the Fourier expansion, which recovers exactly what we expect.
        We defined $F$ as the limit of its smooth parts, so therefore any samples of $F$ are almost surely continuous.
    \end{proof}
    
    At the start of this section, we defined $F(t)$ as the limit as $y \rightarrow \infty$ of the smooth sum
    \begin{align*}
        \frac{\eta}{\pi} \sum_{\substack{n \neq 0 \\ P^+(|n|) \leq y}} \frac{1 - e(nt)}{n}X_n.
    \end{align*}
    Since the Fourier coefficients are bounded by $O(1/n)$ and $F$ is almost surely a Fourier series of a continuous function, all finite Fourier sums converge to $F$ uniformly \cite{stackexchange}. Consequently, we can also define the process $F$ as the limit as $N \rightarrow \infty$ of the partial symmetrical sums
    \begin{align*}
        \frac{\eta}{\pi}\sum_{\substack{|n| \leq N}} \frac{1 - e(nt)}{n}X_n.
    \end{align*}
    Therefore for the rest of the paper we can interchangeably use either definition for the infinite series $F(t)$.

\section{Convergence of Finite-Dimensional Distributions of $\mathcal{F}_q$}\label{s: proof of main thm}

    In order to prove Theorem \ref{thm: main theorem}, we will first show convergence of finite-dimensional distributions. We take $(\mathcal{F}_{q,\pm}(t))_{q \text{ prime}}$ as the sequence of distributions of character paths modulo $q$ dependent on the parity of the characters. Let $F_{\pm}(t)$ be random processes defined by
    \begin{align*}
            &F_+(t) \coloneqq \frac{\eta}{\pi} \sum_{k \geq 1} \frac{X_k}{k}\sin(2\pi kt) &\text{ and }&
            &F_-(t)\coloneqq \frac{\eta}{\pi} \sum_{k \geq 1} \frac{X_k}{k}(1 - \cos(2\pi kt)),
    \end{align*}
    where $X_n$ are Steinhaus random multiplicative functions, defined in Definition \ref{def: Steinhaus rmf}, and $\eta$ is uniformly distributed on the unit circle.
    
    \begin{theorem}\label{thm: finite distributions theorem}
        Let $q$ be an odd prime. The sequence of the distributions of character paths $(\mathcal{F}_{q,\pm}(t))_{q}$ converges to the process $F_{\pm}(t)$ in the sense of convergence of finite distributions.
        In other words, for every $n \geq 1$ and for every $n$-tuple $0 \leq t_1 < \cdots < t_n \leq 1$, the vectors 
        \begin{align*}
            \left( \mathcal{F}_{q,\pm}(t_1), \dots, \mathcal{F}_{q,\pm}(t_n) \right)
        \end{align*}
        converge in law as $q \rightarrow \infty$ through the primes to
        \begin{align*}
            \left( F_{\pm}(t_1), \dots, F_{\pm}(t_n) \right).
        \end{align*}
    \end{theorem}
    
    We prove this by the method of moments. We will define a moment $M_{q}$, of our distribution $\mathcal{F}_{q}$ and a moment $M$ for the random process $F$. 
    In Section \ref{s: method of moments}, we prove $M$ is determinate. Subsequently, in Section \ref{s: convergence of moments}, we prove this sequence of moments $M_q$ converges to $M$, the moment of $F$.
    This is sufficient to prove Theorem \ref{thm: finite distributions theorem}. 
    
    We are considering odd and even character paths separately.
    In this proof we will focus on results for odd character paths as the proof is analogous for the even character case. Where this is not the case, any changes will be addressed throughout the section.  
    
\subsection{Definitions of the Moments} \label{s: defn of moments}
    
    The Fourier series of the character path is 
    \begin{align*}
        f_\chi(t) = \frac{\tau(\chi)}{2 \pi i \sqrt{q}} \sum_{0 < |k| < q} \frac{\overline{\chi}(k)}{k} (1 - e(-kt)) + O\left( \frac{\log q}{\sqrt{q}} \right).
    \end{align*}
    This results from truncating the Fourier series of the character sum $S_\chi(t)$ and the trivial inequality $|f_\chi(t) - S_\chi(t)| \leq \frac{1}{\sqrt{q}}$. The paths of odd and even characters are shown to differ greatly, exemplified in Figure \ref{fig: mod10007}, due to the constant term vanishing when $\chi$ is even. As such, this paper will assess distributions of these character paths modulo odd prime $q$ separately, dependent on parity. As a Fourier series we split this into
    \begin{align*}
             f_\chi(t) = \bigg \{
             \begin{array}{cl} \frac{- \tau(\chi)}{\pi\sqrt{q}} \sum_{k = 1}^q \frac{\overline{\chi}(k)}{k}\sin(2\pi kt) + O\left( \frac{\log q}{\sqrt{q}}\right),  &  \chi \text{ even},\\ \frac{\tau(\chi)}{\pi i \sqrt{q}} \sum_{k =1}^q \frac{\overline{\chi}(k)}{k}(1 - \cos(2\pi kt)) + O\left( \frac{\log q}{\sqrt{q}}\right),  & \chi \text{ odd}. \end{array}
    \end{align*}
    
    We define our moments $M_q$ and $M$. In this section we will assume $\chi$ is odd as the proof is analogous to the even case. Therefore, taking a function from the odd distribution $\mathcal{F}_{q,-}$, we will take the character path modulo $q$ as
    \begin{align*}
        f_\chi(t) = \frac{\tau(\chi)}{\pi i\sqrt{q}} \sum_{1 \leq n \leq q} \frac{\overline{\chi}(n)}{n} (1 - \cos(2 \pi n t)) + O \left( \frac{\log q}{\sqrt{q}}\right).
    \end{align*}
    We will also be considering the odd random series
    \begin{align*}
        F_-(t) = \frac{\eta}{\pi} \sum_{n \geq 1} \frac{X_k}{k}(1 - \cos(2\pi nt)),
    \end{align*}
    which for ease of notation will be referred to as $F(t)$ for the rest of this section.

    \begin{definition}\label{def: Mq}
    Let $k \geq 1$ be given and $\underline{t} = (t_1, \dots, t_k)$, where $0 \leq t_1 < \dots < t_k \leq 1$, be fixed. Additionally fix $\underline{n} = (n_1, \dots, n_k)$ and $\underline{m} = (m_1, \dots, m_k)$, where $n_i,m_i \in \mathbb{Z}_{\geq 0}$. We define the moment sequence $M_q(\underline{n},\underline{m})$ as
        \begin{align*}
            M_q(\underline{n},\underline{m}) = \frac{2}{\phi(q)} \sum_{\chi \odd} \prod_{i=1}^k f_\chi(t_i)^{n_i} \overline{f_\chi(t_i)}^{m_i},
        \end{align*}
    and the moment $M(\underline{n},\underline{m})$ as 
    \begin{align*}
         M(\underline{n},\underline{m}) = \mathbb{E}\left(\prod_{i=1}^k F(t_i)^{n_i} \overline{F(t_i)}^{m_i}\right).
    \end{align*}
    \end{definition}
    
    The moment $M(\underline{n},\underline{m})$ is well defined. To show this we prove the equivalent result that $\mathbb{E}(|F(t)|^n)$ is bounded for all $n$. By Fubini's theorem,
    \begin{align*}
        \mathbb{E}\left(|F(t)|^n\right) = \int_0^\infty n x^{n-1} \mathbb{P}\left( |F(t)| > x\right) dx.
    \end{align*}
    We then use a result by Bober, Goldmakher, Granville and Koukoulopoulos \cite{bggk2014}: Let $c = e^{-\gamma}\log 2$. For any $\tau \geq 1$,
    \begin{align*}
        \mathbb{P}\left( \max_{0 \leq t \leq 1} \left| F(t)\right| > 2 e^\gamma \tau \right) \leq \exp \left\{ - \frac{e^{\tau - c - 2}}{\tau} \left( 1 + O \left(\frac{\log \tau}{\tau} \right) \right) \right\}.
    \end{align*}
    Therefore, combining both equations, the moment is finite and well defined. 
    
    $F(t)$ is a random process, defined by the almost surely converging sum
    \begin{align*}
        F(t) = \frac{\eta}{\pi}\sum_{a \geq 1} \frac{X_a}{a}(1 - \cos(2 \pi a t)).
    \end{align*}
    
    As shown in Section \ref{s: Properties of F} we can define $F$ as the limit of the symmetric partial sums. The infinite series $F$ is not absolutely convergent, so justification is needed to manipulate the product $\prod_{i=1}^k F(t_i)^{n_i} \overline{F(t_i)}^{m_i}$. 
    
    We write the expansion of $F(t_i)^{n_i}$ as
    \begin{align*}
       \frac{\eta^{n_i}}{\pi^{n_i}} \sum_{a_{i,1}, \dots, a_{i,n_i} \geq 1} \prod_{j = 1}^{n_i} \frac{X_{a_{i,j}}}{a_{i,j}} (1 - \cos(2 \pi a_{i,j}t_i)),
    \end{align*}
    and $\overline{F(t_i)}^{m_i}$ in a similar manner. Without changing the order of summation, the product $\prod_{i=1}^k F(t_i)^{n_i} \overline{F(t_i)}^{m_i}$ is therefore
    \begin{align*}
             \frac{\eta^{n}\overline{\eta}^m}{\pi^{n+m}}\sum \cdots
            \sum \prod_{i=1}^k \prod_{j=1}^{n_i} \prod_{j'=1}^{m_i} \frac{X_{a_{i,j}}\overline{X_{b_{i,j'}}}}{a_{i,j}b_{i,j'}}(1 - \cos(2 \pi a_{i,j} t_i))(1 - \cos(2 \pi b_{i,j'} t_i)),
    \end{align*}
    where $n = |\underline{n}|$ and $m = |\underline{m}|$ as above. The sums are over $a_{i,j} \geq 1$ and $b_{i,j'} \geq 1$, where $j \in [1,n_i]$ and $j' \in [1,m_i]$ for $i \in [1,k]$.
   
    The moment $M(\underline{n},\underline{m})$ is the expectation of this multivariate sum. To simplify the equation, we want to swap the order of expectation with the order of summation. Since the moment is finite, we use Lebesgue's dominated convergence theorem \cite[Chapter 2, Corollary 5.3]{gut2013} to bring the expectation inside the sum. Using the multiplicativity of Steinhaus random multiplicative functions, the moment $M$ therefore equals\footnote{Note the moment will be different for $F_+$, where we have $\sin(2 \pi a t)$ instead of $(1 - \cos(2 \pi a t))$.} 
    \begin{align*}
        \mathbb{E}\left( \frac{\eta^{n}\overline{\eta}^m}{\pi^{n+m}}\right) \sum \cdots \sum\mathbb{E}\left(X_{a}\overline{X_{b}}\right)  \prod_{i=1}^k \prod_{j=1}^{n_i} \prod_{j'=1}^{m_i} \frac{(1 - \cos(2 \pi a_{i,j} t_i))(1 - \cos(2 \pi b_{i,j'} t_i))}{a_{i,j}b_{i,j'}},
    \end{align*}
    where 
    \begin{align*}
        \begin{array}{cc}
             a \coloneqq \prod_{i=1}^k \prod_{j=1}^{n_i} a_{i,j}, & b \coloneqq \prod_{i=1}^k \prod_{j' = 1}^{m_i} b_{i,j'}. 
        \end{array}
    \end{align*}
    
    Steinhaus random multiplicative functions $X_n$ are orthogonal as $n$ can always be written as a unique prime factorisation and $\mathbb{E}(X_p) = 0$ for all primes $p$. In other words, 
        \begin{align*}
            \mathbb{E}\left( X_{a} \overline{X_{b}}\right) = \mathbbm{1}_{a = b} \coloneqq \bigg\{ \begin{array}{cl}
                1, & a = b \\
                0, & \text{otherwise}
            \end{array}.
        \end{align*}
        
    Therefore we can rewrite the moment as follows,
    \begin{align*}
         M = \mathbb{E}\left( \frac{\eta^{n}\overline{\eta}^m}{\pi^{n+m}}\right)
         \sum_{l = 1}^\infty \sum_{a = b = l} \frac{1}{ab} \prod_{i=1}^k \prod_{j=1}^{n_i} \prod_{j'=1}^{m_i} (1 - \cos(2 \pi a_{i,j} t_i))(1 - \cos(2 \pi b_{i,j'} t_i)),
    \end{align*}
    where $a$ and $b$ are the product of $a_{i,j}$ and $b_{i,j'}$ respectively. Taking $\frac{1}{ab} = \frac{1}{l^2}$ and bounding $(1 - \cos(x)) \leq 2$, $M$ is clearly bounded as a function of $n_i$ and $m_i$. These variables are fixed and finite, so the moment is absolutely convergent. Therefore, we can swap the order of summation. As a result,
    \begin{align}\label{eq: limiting moment}
        M(\underline{n},\underline{m}) = \mathbb{E}\left( \frac{\eta^{n}\overline{\eta}^m}{\pi^{n+m}}\right) \sum_{a \geq 1} \mathcal{B}_{\underline{n},\underline{t}}(a) \mathcal{B}_{\underline{m},\underline{t}}(a),
    \end{align}
    where 
    \begin{align}\label{eq: mathcal B for limiting moment}
        \mathcal{B}_{\underline{N},\underline{t}}(x) \coloneqq \sum_{x_1 \cdots x_k = x} \prod_{i=1}^k \beta_{N_i,t_i}(x_i)
    \end{align}
    and 
    \begin{align}\label{eq: beta for limiting moment}
        \beta_{N_i,t_i}(x_i)= \sum_{y_1 \cdots y_{N_i} = x_i} \frac{1}{x_i}\prod_{j=1}^{N_i} (1 - \cos(2 \pi y_j t_i)).
    \end{align}

    The moment $M_q(\underline{n},\underline{m})$ can be also be rewritten using methods from Bober and Goldmakher \cite{bobergoldmakher2013}. First, we use the Fourier expansion of $f_\chi(t)$, so 
    \begin{align*}
            f_\chi(t_i)^{n_i}\overline{f_\chi(t_i)}^{m_i}
            &= \frac{\tau(\chi)^{n_i}\overline{\tau(\chi)}^{m_i}}{(\pi \sqrt{q} )^{n_i + m_i}i^{n_i - m_i}} \sum_{\substack{1 \leq a \leq q^{n_i} \\ 1 \leq b \leq q^{m_i}}} \overline{\chi}(a)\chi(b)\beta_{n_i,q,t_i}(a) \beta_{m_i,q,t_i}(b)
            + O\left( \frac{(\log q + 1)^{n_i + m_i}}{\sqrt{q}}\right),
        \end{align*}
        where $\beta_{N,q,t}$ is defined as
    \begin{align}
        \beta_{N,q,t}(x) \coloneqq \sum_{\substack{ x_1 \cdots x_{N} = x \\ x_i \leq q}} \frac{1}{x}\prod_{k = 1}^{N} ( 1- \cos(2 \pi x_k t) ),
    \end{align}
    for $(x,q) = 1$ and 0 otherwise\footnote{For even characters, $\beta_{N,q,t}$ would instead sum over the product of $\sin(2\pi x_k t)$.}.

   Continuing to expand $M_q(\underline{n},\underline{m})$, we take a product of all $ f_\chi(t_i)^{n_i}\overline{f_\chi(t_i)}^{m_i}$ for $i \in [1,k]$, showing
    \begin{align*}
         \prod_{i = 1}^k f_\chi(t_i)^{n_i}\overline{f_\chi(t_i)}^{m_i}
        = \frac{\tau(\chi)^{n}\overline{\tau(\chi)}^{m}}{(\pi \sqrt{q} )^{n + m}i^{n-m}} \sum_{\substack {1 \leq a \leq q^{n} \\ 1 \leq b \leq q^{m} \\ \text{ for } i \in [1,k]}} \overline{\chi}(a)\chi(b) \mathcal{B}_{\underline{n},q,\underline{t}}(a) \mathcal{B}_{\underline{m},q,\underline{t}}(b) + O \left( \frac{(\log q)^{n+m}}{\sqrt{q}} \right),
    \end{align*}
    where 
    $$n \coloneqq n_1 + \cdots + n_k \text{ and }m\coloneqq m_1 + \cdots + m_k$$
    and
    \begin{align}\label{eq: definition of mathcal b}
        \mathcal{B}_{\underline{N},q,\underline{t}}(x) \coloneqq \sum_{\substack{x_1 \cdots x_k = x \\ x_i \leq q^{N_i}}} \prod_{i=1}^k \beta_{N_i,q,t_i}(x_i),
    \end{align}
    for $(x,q) = 1$ and 0 otherwise.
    Note that $\mathcal{B}_{\underline{N},\underline{t}}$ and $\beta_{N_i,t_i}$ from Equations \eqref{eq: mathcal B for limiting moment} and \eqref{eq: beta for limiting moment} are the limits as $q \rightarrow \infty$ of $\mathcal{B}_{\underline{N},q,\underline{t}}$ respectively.
    Furthermore, we take the average of this product over all odd Dirichlet characters $\chi$ to find 
    \begin{align}\label{eq: Mq}
            M_q(\underline{n},\underline{m}) =& \frac{1}{(\pi \sqrt{q})^{n + m}i^{n-m}} \sum_{\substack {1 \leq a \leq q^{n} \\ 1 \leq b \leq q^{m} }}\left( \mathcal{B}_{\underline{n},q,\underline{t}}(a)\mathcal{B}_{\underline{m},q,\underline{t}}(b) \right) \frac{2}{\phi(q)} \sum_{\substack{\chi \mod q \\ \chi \odd}} \overline{\chi}(a)\chi(b) \tau(\chi)^{n} \overline{\tau(\chi)}^m\\
             &+ O \left( \frac{(\log q)^{n+m}}{\sqrt{q}} \right). \nonumber
        \end{align}
     This form is more useful for future calculations and will used to prove $M_q$ tends to $M$ as $q \rightarrow \infty$ through the primes.

\subsection{Bounding the Moments}\label{s: bounding the moments}

    Later in the paper we will be interested in bounding $\mathcal{B}_{\underline{N},q,\underline{t}}$ and $\mathcal{B}_{\underline{N},\underline{t}}$. The inequality we find is independent of $q$, so we can consider both bounds at the same time. Therefore for this subsection we will work with $\mathcal{B}_{\underline{N},q,\underline{t}}$. 
    
    Recall,
    \begin{align*}
        \mathcal{B}_{\underline{N},q,\underline{t}}(x) &= \sum_{\substack{x_1 \cdots x_k = x \\ x_i \leq q^{N_i}}} \prod_{i=1}^k \beta_{N_i,q,t_i}(x_i),
    \end{align*}
    where
    \begin{align*}
        \beta_{N,q,t}(x_i) = \sum_{\substack{y_1 \cdots y_N = x_i \\ y_j \leq q}} \frac{1}{x_i} \prod_{j=1}^N (1-\cos(2 \pi y_j t)),
    \end{align*}
    for $(x_i,q) = 1$ and 0 otherwise.
    Since $|1 - \cos(2 \pi y_j t)| \leq 2$, we always have the bound
    \begin{align*}
         | \beta_{N,q,t}(x) | \leq \frac{2^N d_N(x)}{x},
    \end{align*}
    where $d_N(x)$ is the $N$th divisor function $\sum_{x_1 \cdots x_N = x} 1$\footnote{For $F_+$ and even character paths, the $2^N$ vanishes in the bound of $\beta$ as $|\sin(2 \pi y_j t)| \leq 1$. However, since this bound is only included in error terms, the difference of the constant is irrelevant.}. As a result,
    \begin{align*}
        \mathcal{B}_{\underline{N},q,\underline{t}}(x) \leq \frac{2^N}{x} \sum_{\substack{x_1 \cdots x_k = x \\ x_i \leq q^{N_i}}} \prod_{i=1}^k d_{N_i}(x_i),
    \end{align*}
    where $N = \sum N_i = |\underline{N}|$. To further bound $\mathcal{B}$ we next use the following lemma.
    \begin{lemma}\label{thm: divisor multiplying}
        Let $d_{N_1}(x_1),d_{N_2}(x_2)$ be the $N_1$th and $N_2$th divisor function of $x_1,x_2 \in \mathbb{N}$ respectively. We have the relation
        \begin{align*}
            d_{N_1}(x_1) d_{N_2}(x_2) \leq d_{N_1 + N_2}(x_1 \cdot x_2).
        \end{align*}
    \end{lemma}
    \begin{proof}
        We apply a combinatorial argument, where we view $d_{N}(x)$ as the number of ways of choosing $N$ positive integers that multiply to $x$. Therefore $d_{N_1+N_2}(x_1 \cdot x_2)$ is at least the number of ways of choosing $N_1$ integers multiplying to $x_1$ times the number of ways of choosing $N_2$ integers multiplying to $x_2$.
    \end{proof}

    Using Lemma \ref{thm: divisor multiplying}, we bound $\mathcal{B}_{\underline{N},q,\underline{t}}(x)$ by
    \begin{align}
    \mathcal{B}_{\underline{N},q,\underline{t}}(x) &\leq \frac{2^N d_N(x)}{x}\sum_{\substack{x_1\cdots x_k = x \\ x_i \leq q^{N_i}}} 1 
    \leq \frac{2^N d_N(x)d_k(x)}{x} \label{eq: bound of mathcal B} \\
    &\leq \frac{2^N d^{N+k}(x)}{x}. \nonumber
    \end{align}
    In parts of the proof, it is sufficient to use the looser bound, however we will mainly apply the bound from Equation \eqref{eq: bound of mathcal B}.
    This will be useful in future equations. Note that this is independent of $q$ and $\underline{t}$, so the bounds hold for $\mathcal{B}_{\underline{N},\underline{t}} = \lim_{q \rightarrow \infty} \mathcal{B}_{\underline{N},q,\underline{t}}$.
    
\subsection{Proving Determinacy}\label{s: method of moments}

     Our aim is to use the method of moments to prove the distribution of character paths modulo $q$ converges to $F(t)$ in the sense of finite distributions. For this we need to show the moment $M(\underline{n},\underline{m})$ is determinate, or in other words show the moment only has one representing measure. 
     To show that $M$ is a determinate complex moment sequence, it is sufficient to show that it satisfies
    \begin{align}\label{eq: carlemanconditioncomplex}
        \sum_{n=1}^\infty M(\underline{n},\underline{n})^{-1/2n} = \infty,
    \end{align}
    where $\underline{n} = (n_1, \dots, n_k)$ and $n = |\underline{n}| = \sum_i n_i$. This is also known as the Carleman condition \cite[Theorem 15.11]{schmudgen2017}.
    
    \begin{lemma}
        The moment $M(\underline{n},\underline{m})$ satisfies Equation \eqref{eq: carlemanconditioncomplex}.
    \end{lemma}
    
    \begin{proof}
        This is shown using Equation \eqref{eq: limiting moment} and taking $\underline{n} = \underline{m}$. Setting $n = |\underline{n}| = |\underline{m}|$, we have
        \begin{align*}
            M(\underline{n},\underline{m}) = \frac{1}{\pi^{2n}} \sum_{a \geq 1} \mathcal{B}_{\underline{n},\underline{t}}(a)^2.
        \end{align*}
    We use the bound of $\mathcal{B}$ from Equation \eqref{eq: bound of mathcal B}, taking $d_k(a) \leq a^{\epsilon_k}$ for small $\epsilon_k > 0$, so
    \begin{align*}
        M(\underline{n},\underline{n}) \leq \frac{2^{2n}}{\pi^{2n}}\sum_{a\geq 1} \frac{d_n(a)^2}{a^{2-2\epsilon_k}} =: \frac{2^{2n}}{\pi^{2n}}\sum_{a\geq 1} \frac{d_n(a)^2}{a^{2\sigma}},
    \end{align*}
    taking $\sigma \coloneqq 1 - \epsilon_k$.
    We can use Proposition 3.2 from Bober and Goldmakher \cite{bobergoldmakher2013}, which states for $1/2 < \sigma \leq 1$,
    \begin{align}\label{thm: proposition 3.2 from Bober and Goldmkaher}
        \sum_{a=1}^\infty \frac{d_n(a)^2}{a^{2\sigma}} \leq \exp\left( 2 n \sigma \log\log(2n)^{1/\sigma} + \frac{(2n)^{1/\sigma}}{2\sigma - 1} + O\left( \frac{n}{2\sigma - 1} + \frac{(2n)^{1/\sigma}}{\log(3(2n)^{1/\sigma - 1})}\right) \right).
    \end{align}
    
    Here we have shown the sum has the lower bound
    \begin{align*}
        \sum_{n=1}^\infty M(\underline{n},\underline{n})^{-1/2n} \geq \frac{\pi}{2}\sum_{n=1}^\infty \exp\left( - \sigma \log\log\left((2n)^{1/\sigma}\right) - \frac{(2n)^{1/\sigma-1}}{2\sigma - 1} + O\left( \frac{1}{2\sigma - 1} + \frac{(2n)^{1/\sigma-1}}{\log(3(2n)^{1/\sigma - 1})}\right) \right).
    \end{align*}
    The lower bound can be rewritten as
    \begin{align*}
        \frac{\pi}{2}\sum_{n=1}^\infty \frac{\sigma^\sigma}{(\log 2n)^\sigma} \exp\left( - \frac{(2n)^{\frac{1 - \sigma}{\sigma}}}{2\sigma - 1}\right) \exp \left(O\left( \frac{1}{2\sigma - 1} + \frac{(2n)^{1/\sigma-1}}{\log(3(2n)^{1/\sigma - 1})}\right) \right).
    \end{align*}
    Tending $\sigma = 1 - \epsilon_k$ to 1, this sum diverges. Therefore
    \begin{align*}
         \sum_{n=1}^\infty M(\underline{n},\underline{n})^{-1/2n} = \infty,
    \end{align*}
    and the Carleman condition holds. Therefore the claim is proved.
    \end{proof}

\subsection{Convergence of Moments}\label{s: convergence of moments}

        In this section we show the moment sequence $M_q$ converges to the multivariate moment of $F$, therefore proving Theorem \ref{thm: finite distributions theorem}. Separating the distribution by parity, we have two propositions.
        
    \begin{proposition}\label{thm: moment convergence}
         Let $k \geq 1$ be given and $0 \leq t_1 < \dots < t_k \leq 1$ be fixed. Fix $\underline{n} = (n_1, \dots, n_k)$ and $\underline{m} = (m_1, \dots, m_k)$, where $n_i,m_i \in \mathbb{Z}_{\geq 0}$. Let
        \begin{align*}
            M_{q,-}(\underline{n},\underline{m}) = \frac{2}{\phi(q)} \sum_{\chi \odd} \prod_{i=1}^k f_\chi(t_i)^{n_i} \overline{f_\chi(t_i)}^{m_i}.
        \end{align*}
        Then for all $\epsilon > 0$,
        \begin{align*}
            M_{q,-}(\underline{n},\underline{m}) = M_-(\underline{n},\underline{m}) + O_{\underline{n},\underline{m},k}\left( q^{-1/2 + \epsilon} \right),
        \end{align*}
        where
        \begin{align*}
            M_-(\underline{n},\underline{m}) = \mathbb{E}\left( \prod_{i=1}^k F_-(t_i)^{n_i} \overline{F_-(t_i)}^{m_i} \right).
        \end{align*}
        Importantly, $M_{q,-}(\underline{n},\underline{m}) \rightarrow M_-(\underline{n},\underline{m})$ as $q \rightarrow \infty$ through the primes.
    \end{proposition}
    
    \begin{proposition}\label{thm: moment convergence even}
         Let $k \geq 1$ be given and fix $\underline{t},\underline{n},\underline{m}$ as in Proposition \ref{thm: moment convergence}. Let
        \begin{align*}
            M_{q,+}(\underline{n},\underline{m}) = \frac{2}{\phi(q)} \sum_{\chi \even} \prod_{i=1}^k f_\chi(t_i)^{n_i} \overline{f_\chi(t_i)}^{m_i}.
        \end{align*}
        Then for all $\epsilon > 0$,
        \begin{align*}
            M_{q,+}(\underline{n},\underline{m}) = M_-(\underline{n},\underline{m}) + O_{\underline{n},\underline{m},k}\left( q^{-1/2 + \epsilon} \right),
        \end{align*}
        where
        \begin{align*}
            M_+(\underline{n},\underline{m}) = \mathbb{E}\left( \prod_{i=1}^k F_+(t_i)^{n_i} \overline{F_+(t_i)}^{m_i} \right).
        \end{align*}
        Importantly, $M_{q,+}(\underline{n},\underline{m}) \rightarrow M_+(\underline{n},\underline{m})$ as $q \rightarrow \infty$ through the primes.
    \end{proposition}
    
    In this section we only look at the $\mathcal{F}_{q,-}$ case, where $\chi$ is odd. There are equivalent propositions and lemmas for the even case, where the proofs are analogous to the proofs shown in the section. In places where the proof differs, we will state the results for $\mathcal{F}_{q,+}$ and how it doesn't largely affect the proof.
    
    These propositions are sufficient to prove Theorem \ref{thm: finite distributions theorem}, showing $(\mathcal{F}_q(t))_{q \text{ prime}}$ converges in finite distributions to $F(t)$. We prove Proposition \ref{thm: moment convergence} as a combination of the following 2 propositions.
    
    \begin{proposition}\label{thm: prop moment seq to sum}
        Let $k \geq 1$ be given and $\underline{t} = (t_1, \dots, t_k)$, where $0 \leq t_1 < \dots < t_k \leq 1$, be fixed. Fix $\underline{n} = (n_1, \dots, n_k)$ and $\underline{m} = (m_1, \dots, m_k)$, where $n_i,m_i \in \mathbb{Z}_{\geq 0}$ and 
        $$n \coloneqq n_1 + n_2 + \cdots + n_k = m_1 + \cdots + m_k.$$ 
        The moment sequence defined in Proposition \ref{thm: moment convergence} can be expressed as
        \begin{align*}
            M_{q,-}(\underline{n},\underline{m}) = \frac{1}{\pi^{2n}} \sum_{a \geq 1} \mathcal{B}_{\underline{n},\underline{t}}(a) \mathcal{B}_{\underline{m},\underline{t}}(a)+ \mainerror,
        \end{align*}
       where $\mathcal{B}_{\underline{N},\underline{t}}$ is defined as
        \begin{align*}
            \mathcal{B}_{\underline{N},\underline{t}}(a) = \frac{1}{a} \sum_{x_1 \cdots x_k = a} \prod_{i=1}^k \left( \sum_{\substack{ y_1 \cdots y_{N} = x_i}} \prod_{j = 1}^{N_i} ( 1- \cos(2 \pi y_j t) )\right).
        \end{align*}
    \end{proposition}
    
    \begin{proposition}\label{thm: prop expectation to sum}
        Let $k \geq 1$ be given and $\underline{t} = (t_1, \dots, t_k)$, where $0 \leq t_1 < \dots < t_k \leq 1$, be fixed. Fix $\underline{n} = (n_1, \dots, n_k)$ and $\underline{m} = (m_1, \dots, m_k)$, where $n_i,m_i \in \mathbb{Z}_{\geq 0}$ and 
        $$n \coloneqq n_1 + n_2 + \cdots + n_k = m_1 + \cdots + m_k.$$ 
        Then 
        \begin{align*}
            M_-(\underline{n},\underline{m}) = \frac{1}{\pi^{2n}} \sum_{a \geq 1} \mathcal{B}_{\underline{n},\underline{t}}(a) \mathcal{B}_{\underline{m},\underline{t}}(a),
        \end{align*}
       where $\mathcal{B}_{\underline{N},\underline{t}}$ is defined as in Proposition \ref{thm: prop moment seq to sum}.
    \end{proposition}
    
    Before the proof of the propositions, we will use them to prove Proposition \ref{thm: moment convergence}.
    
    \begin{proof}[Proof of Proposition \ref{thm: moment convergence}]
    
        Take $n = n_1 + n_2 + \cdots + n_k$ and $m = m_1 + \cdots m_k$. We split the proof into 2 cases: $n = m$ and $n \neq m$. The first case has already been shown by Propositions \ref{thm: prop moment seq to sum} and \ref{thm: prop expectation to sum}:
        \begin{align*}
            M_{q,-}(\underline{n},\underline{m}) &= \frac{1}{\pi^{2n}} \sum_{a \geq 1} \mathcal{B}_{\underline{n},\underline{t}}(a) \mathcal{B}_{\underline{m},\underline{t}}(a) + \mainerror \\
            &= \mathbb{E}\left( \prod_{i=1}^k F(t_i)^{n_i} \overline{F(t_i)}^{m_i} \right) + \mainerror.
        \end{align*}
        
        Therefore, the only case left to show is when $n \neq m$. We recall Equation \eqref{eq: limiting moment}:
        \begin{align*}
            M_-(\underline{n},\underline{m}) = \mathbb{E}\left(\frac{\eta^{n}\overline{\eta}^m}{\pi^{n+m}}\right) \sum_{a \geq 1} \mathcal{B}_{\underline{n},\underline{t}}(a) \mathcal{B}_{\underline{m},\underline{t}}(a).
        \end{align*}
        Since $\eta$ is uniformly distributed on the unit circle, $\mathbb{E}(\eta^n\overline{\eta}^m) = 0$ and the moment $M_-$ vanishes. Therefore, to conclude the proof, we need to show the moment $M_{q,-}\rightarrow 0$ as $q \rightarrow \infty$. As shown in Equation \eqref{eq: Mq}, we can write $M_{q,-}(\underline{n},\underline{m})$ as
        \begin{align*}
            M_{q,-}(\underline{n},\underline{m}) =& \frac{1}{(\pi \sqrt{q})^{n + m}} \sum_{\substack {1 \leq a \leq q^{n} \\ 1 \leq b \leq q^{m} }}\left( \mathcal{B}_{\underline{n},q,\underline{t}}(a)\mathcal{B}_{\underline{m},q,\underline{t}}(b) \right) \frac{2}{\phi(q)} \sum_{\chi} \overline{\chi}(a)\chi(b) \tau(\chi)^{n} \overline{\tau(\chi)}^m\\
             &+ O \left( \frac{(\log q)^{n+m}}{\sqrt{q}} \right).
        \end{align*}
        Assuming $n > m$, we rewrite $\tau(\chi)^n\overline{\tau(\chi)}^m$ as $q^m \tau(\chi)^{n-m}$. Therefore, taking $\chi(\overline{a}) \coloneqq \overline{\chi}(a)$,
        \begin{align*}
             \frac{2}{\phi(q)} \sum_{\chi} \overline{\chi}(a)\chi(b) \tau(\chi)^{n} \overline{\tau(\chi)}^m = \frac{2q^m}{\phi(q)} \sum_{\chi} \chi(\overline{a}\cdot b) \tau(\chi)^{n-m}.
        \end{align*}
       \begin{lemma}\label{thm: granvillesoundlemma}
            For $N \in \mathbb{N}$,
            \begin{align*}
        \frac{2}{\phi(q)} \bigg | \sum_{\substack{\chi \mod q \\ \chi(-1) = \sigma}} \chi(a) \tau(\chi)^N \bigg | \leq 2N q^{(N-1)/2},
            \end{align*}
            where $\sigma =\{1,-1\}$.
        \end{lemma}    
       This lemma is a slight generalisation of a result by Granville and Soundararajan \cite[Lemma 8.3]{granvillesoundararajan2007} and uses Deligne's bound on hyper-Kloosterman sums. Below follows Granville and Soundararajan's proof, with a modification to include when $\chi$ is even.
        \begin{proof} 
        Firstly, we rewrite the sum as exponential sums, using orthogonality of characters and the definition of the Gauss sum $\tau(\chi)$:
        \begin{align*}
            \frac{2}{\phi(q)} \sum_{\substack{\chi \mod q \\ \chi(-1) = \sigma}} \chi(a) \tau(\chi)^N = \sum_{\substack{x_1, \dots, x_N\mod q \\ x_1 \cdots x_N \equiv \overline{a} \mod q}} e\left( \frac{x_1 + \cdots + x_N}{q}\right) + \text{sgn}(\sigma) \sum_{\substack{x_1, \dots, x_N\mod q \\ x_1 \cdots x_N \equiv -\overline{a} \mod q}} e\left( \frac{x_1 + \cdots + x_N}{q}\right).
        \end{align*}
        Then, using Deligne's bound \cite{deligne1977}
        \begin{align*}
           \bigg | \sum_{\substack{x_1, \dots, x_N \mod q \\ x_1 \cdots x_N \equiv b \mod q}} e\left( \frac{x_1 + \cdots + x_N}{q}\right) \bigg | \leq N q^{(N-1)/2},
        \end{align*}
        we have proved the lemma.
        \end{proof}
       
       As a result, we have the inequality
       \begin{align*}
           |M_{q,-}(\underline{n},\underline{m})| \leq& \frac{2(n-m)}{\pi^{(n + m)}\sqrt{q}} \sum_{\substack {1 \leq a \leq q^{n} \\ 1 \leq b \leq q^{m}}} |\mathcal{B}_{\underline{n},q,\underline{t}}(a)||\mathcal{B}_{\underline{m},q,\underline{t}}(b)| + O \left( \frac{(\log q)^{n+m}}{\sqrt{q}} \right).
       \end{align*}
        
        We also have the bound on $\mathcal{B}$, as shown in Equation \eqref{eq: bound of mathcal B},
        \begin{align*}
             \mathcal{B}_{\underline{N},q,\underline{t}}(x) \leq \frac{2^N d_N(x)d_k(x)}{x}.
        \end{align*}
        Therefore, trivially bounding both divisor functions by $q^\epsilon$ for $\epsilon > 0$,
        \begin{align*}
            \sum_{\substack {1 \leq a \leq q^{n}}} |\mathcal{B}_{\underline{n},q,\underline{t}}(a)| \ll 2^n q^\epsilon \sum_{1 \leq a \leq q^n} \frac{1}{a} \leq 2^n q^\epsilon \log(q^n).
        \end{align*}
        We get an analogous result for $\sum_{\substack {1 \leq b \leq q^{m}}} |\mathcal{B}_{\underline{m},q,\underline{t}}(b)|$. As a result,
        \begin{align*}
            M_{q,-}(\underline{n},\underline{m}) \ll\frac{2^{1 + n + m}(n-m)q^\epsilon\log(q^n)\log(q^m)}{\pi^{n+m} \sqrt{q}}  + O \left( \frac{(\log q)^{n+m}}{\sqrt{q}} \right) \ll_{n,m} q^{-1/2 + \epsilon},
        \end{align*}
        which tends to zero as $q \rightarrow \infty$. By a similar method we can show this is also the case when $n < m$. Therefore Proposition \ref{thm: moment convergence} holds.
    \end{proof}
    
    Having proven Proposition \ref{thm: moment convergence} assuming Propositions \ref{thm: prop moment seq to sum} and \ref{thm: prop expectation to sum}, we will now prove both results, showing when $|\underline{n}|= |\underline{m}|$ both $\lim_{q \rightarrow \infty} M_q$ and $M$ equal
    \begin{align*}
        \frac{1}{\pi^{2n}} \sum_{a \geq 1} \mathcal{B}_{\underline{n},\underline{t}}(a) \mathcal{B}_{\underline{m},\underline{t}}(a).
    \end{align*}
    Recall that we are only proving the odd character case, as the proof is analogous for $\mathcal{F}_{q,+}$. For odd characters, $\mathcal{B}_{N,q,t}(x)$ is defined in Section \ref{s: defn of moments} as
    \begin{align*}
        \sum_{\substack{x_1 \cdots x_k = x \\ x_i \leq q^{N_i}}} \prod_{i=1}^k \beta_{N_i,q,t_i}(x_i) =  \sum_{\substack{x_1 \cdots x_k = x \\ x_i \leq q^{N_i}}} \frac{1}{x} \prod_{i=1}^k\sum_{\substack{ y_1 \cdots y_{N_i} = x_i \\ y_i \leq q}} \prod_{j = 1}^{N_i} ( 1- \cos(2 \pi y_j t) ).
    \end{align*}
    For ease of notation, in the proofs we refer to $M_{q,-}$ and $M_-$ as $M_q$ and $M$ respectively.
    
     \begin{proof}[Proof of Proposition \ref{thm: prop moment seq to sum}]
         Taking $n = m$, where
         \begin{align*}
             &n \coloneqq |\underline{n}| = n_1 + \cdots + n_k, &m \coloneqq |\underline{m}| = m_1 + \cdots + m_k,
         \end{align*}
         we rewrite Equation \eqref{eq: Mq} to
         \begin{align*}
             M_q\left( \underline{n}, \underline{m}\right)= \frac{1}{\pi^{2n}} \sum_{\substack {1 \leq a, b\leq q^{n}}}\left( \mathcal{B}_{\underline{n},q,\underline{t}}(a) \mathcal{B}_{\underline{m},q,\underline{t}}(b) \right) \frac{2}{\phi(q)} \sum_{\chi \text{ odd}} \overline{\chi}(a)\chi(b) + O\left(\frac{(\log q)^{2n}}{\sqrt{q}}\right),
         \end{align*}
         where $\mathcal{B}_{\underline{n},q,\underline{t}}$ is defined as in Equation \eqref{eq: definition of mathcal b}.
         
         Using the orthogonality of $\chi$, and noting we are only summing over odd characters $\chi$ modulo $q$\footnote{The method for the even character case would differ here. Firstly recall $\mathcal{B}$ involves $\sin$ instead of $(1 - \cos)$ in the even case. Additionally, since these are even characters, we would have $M_q = \frac{1}{\pi^{2n}} \Sigma_+ + \frac{1}{\pi^{2n}}\Sigma_- + O\left( (\log q)^{2n}q^{-1/2} \right)$, where we are adding the $\Sigma_-$ term instead of subtracting it. However, the $\Sigma_-$ term is eventually swallowed by the error term, so this doesn't affect the end result.}, the moment sequence becomes
         \begin{align}\label{eq: moment sequence with the sigmas}
             M_q(\underline{n},\underline{m}) =& \frac{1}{\pi^{2n}}             \Sigma_{+}(\underline{n},\underline{m})
              - \frac{1}{\pi^{2n}} \Sigma_{-}(\underline{n},\underline{m})
             + O \left( \frac{(\log q)^{2n}}{\sqrt{q}} \right),
         \end{align}
         where
         \begin{align*}
             \Sigma_{+}(\underline{n},\underline{m}) \coloneqq \sum_{\substack{1 \leq a,b \leq q^n \\ a \equiv b \mod q}} \mathcal{B}_{\underline{n},q,\underline{t}}(a)\mathcal{B}_{\underline{m},q,\underline{t}}(b),&
             &\Sigma_{-}(\underline{n},\underline{m}) \coloneqq \sum_{\substack{1 \leq a,b \leq q^n \\ a \equiv - b \mod q}} \mathcal{B}_{\underline{n},q,\underline{t}}(a)\mathcal{B}_{\underline{m},q,\underline{t}}(b).
         \end{align*}
         
         The aim is to get the main sum independent of $q$. Using ideas from Bober and Goldmakher \cite[Proof of Lemma 4.1]{bobergoldmakher2013}, we consider $\Sigma_+$ and $\Sigma_-$ simultaneously. First we split the sums into arithmetic progressions mod $q$, noting that $\mathcal{B}_{\underline{n},q,\underline{t}}(x)$ vanishes when $q|x$, 
         \begin{align*}
              \Sigma_{\pm}(\underline{n},\underline{m}) &= \sum_{\substack{1 \leq a,b < q \\ a \equiv \pm b \mod q}} \sum_{0 \leq \gamma_1, \gamma_2 < q^{n-1}} \mathcal{B}_{\underline{n},q,\underline{t}}(a + \gamma_1 q)\mathcal{B}_{\underline{m},q,\underline{t}}(b + \gamma_2 q).
         \end{align*}
        
        We simplify $\Sigma_{\pm}$ by splitting the inner sum into $\gamma_1 = \gamma_2 = 0$, $\gamma_1 \neq 0$, and $\gamma_2 \neq 0$:
        \begin{align*}
            \sum_{0 \leq \gamma_1, \gamma_2 < q^{n-1}} &\mathcal{B}_{\underline{n},q,\underline{t}}(a + \gamma_1 q)\mathcal{B}_{\underline{m},q,\underline{t}}(b + \gamma_2 q) \\
            &= \left( \mathcal{B}_{\underline{n},q,\underline{t}}(a) \mathcal{B}_{\underline{m},q,\underline{t}}(b)\right) + \sum_{j=1}^2 \sum_{\substack{0 \leq \gamma_1, \gamma_2 < q^{n-1}\\\gamma_j \neq 0}} \mathcal{B}_{\underline{n},q,\underline{t}}(a + \gamma_1 q)\mathcal{B}_{\underline{m},q,\underline{t}}(b + \gamma_2 q).
        \end{align*}
    
    For ease of notation, we define the above latter sum as
    \begin{align*}
        \Omega = \sum_{j=1}^2 \sum_{\substack{0 \leq \gamma_1, \gamma_2 < q^{n-1}\\\gamma_j \neq 0}} \mathcal{B}_{\underline{n},q,\underline{t}}(a + \gamma_1 q)\mathcal{B}_{\underline{m},q,\underline{t}}(b + \gamma_2 q).
    \end{align*}
    We can bound $\Omega$ by using the bound of $\mathcal{B}$ shown in Section \ref{s: bounding the moments}, so
    \begin{align*}
       \Omega
        &\leq 2^{2n}\sum_{j=1}^2 \sum_{\substack{0 \leq \gamma_1, \gamma_2 < q^{n-1}\\\gamma_j \neq 0}} \frac{ d_n(a+ \gamma_1 q) d_k(a+\gamma_1 q)}{a + \gamma_1 q}  \frac{d_n(b+\gamma_2 q) d_k(b + \gamma_2 q)}{b + \gamma_2 q} .
    \end{align*}
    By bounding the divisor functions by $O_{n,k}(q^{\epsilon})$, we can further bound the sum to 
    \begin{align*}
          \Omega \ll_{n,k} q^\epsilon \sum_{j=1}^2\sum_{\substack{0 \leq \gamma_1, \gamma_2 < q^{n-1}\\\gamma_j \neq 0}} \frac{1}{a + \gamma_1 q}  \frac{1}{b + \gamma_2 q}.
    \end{align*}
    We can use the bound on partial harmonic series,
    \begin{align*}
        \omega_x \coloneqq \sum_{\gamma = 1}^{q^{n-1}} \frac{1}{x + \gamma q} \leq \frac{\log(q^{n-1})}{q},
    \end{align*}
    to further bound $\Omega$. As a result,
    \begin{align*}
        \sum_{j=1}^2\sum_{\substack{0 \leq \gamma_1, \gamma_2 < q^{n-1}\\\gamma_j \neq 0}} \frac{1}{a + \gamma_1 q}  \frac{1}{b + \gamma_2 q} 
        = \left(\frac{1}{a} + \omega_a \right)\omega_b + \omega_a\left(\frac{1}{b} + \omega_b \right) 
        \leq \frac{\log(q^{n-1})}{q} \left(\frac{1}{a} +\frac{1}{b} + \frac{2\log(q^{n-1})}{q} \right).
    \end{align*}
    Therefore $\Sigma_{\pm}$ can be written as,
    \begin{align*}
        \Sigma_{\pm} 
        &= \sum_{\substack{1 \leq a,b < q \\ a \equiv \pm b \mod q}} \left( \mathcal{B}_{\underline{n},q,\underline{t}}(a) \mathcal{B}_{\underline{m},q,\underline{t}}(b)\right) + O_{n,k} \left(\frac{ q^\epsilon \log(q^{n-1})}{q} \sum_{\substack{1 \leq a,b < q \\ a \equiv \pm b \mod q}} \left( \frac{1}{a} + \frac{1}{b} + \frac{2\log(q^{n-1})}{q}\right)\right).
    \end{align*}
    
    For $\Sigma_+$ we have $a \equiv +b \mod q$ and $1 \leq a,b \leq q$. Therefore $a = b$ and we have
    \begin{align*}
        \Sigma_{+} 
        =& \sum_{\substack{1 \leq a < q}} \left( \mathcal{B}_{\underline{n},q,\underline{t}}(a) \mathcal{B}_{\underline{m},q,\underline{t}}(a)\right) + O_{n,k} \left(\frac{ q^\epsilon \log(q^{n-1})}{q} \sum_{\substack{1 \leq a < q }} \left( \frac{2}{a} + \frac{2\log(q^{n-1})}{q}\right)\right).
    \end{align*}
    For $\Sigma_-$ we have $1\leq a,b \leq q$ and $a \equiv -b \mod q$. Therefore $b = q - a$ and
    \begin{align*}
        \Sigma_{-} 
        =& \sum_{\substack{1 \leq a < q}} \left( \mathcal{B}_{\underline{n},q,\underline{t}}(a) \mathcal{B}_{\underline{m},q,\underline{t}}(q-a)\right) + O_{n,k} \left(\frac{  q^\epsilon \log(q^{n-1})}{q} \sum_{\substack{1 \leq a < q }} \left( \frac{1}{a} + \frac{1}{q-a} + \frac{2\log(q^{n-1})}{q}\right)\right).
    \end{align*}
    We bound the partial harmonic series again by $\log q$ to simplify both errors for $\Sigma_+$ and $\Sigma_-$. Consequently both error terms above can be bounded by
    \begin{align*}
        O_{n,k}\left(\frac{ \log(q^{n-1})}{q^{1-\epsilon}}\left(2\log q + 2\log(q^{n-1})\right)\right).
    \end{align*}
     
    By combining the error terms, the moment sequence from Equation \eqref{eq: moment sequence with the sigmas} is\footnote{Recall we are add the $\Sigma_-$ term instead of subtracting it in the even character case.}
        \begin{align*}
              M_q(\underline{n},\underline{m}) =& \frac{1}{\pi^{2n}}             \Sigma_{+}(\underline{n},\underline{m})
              - \frac{1}{\pi^{2n}} \Sigma_{-}(\underline{n},\underline{m})
             + O \left( \frac{(\log q)^{2n}}{\sqrt{q}} \right) \\
             =& \frac{1}{\pi^{2n}}\sum_{\substack{1 \leq a \leq q}} \left( \mathcal{B}_{\underline{n},q,\underline{t}}(a) \mathcal{B}_{\underline{m},q,\underline{t}}(a)\right) - \frac{1}{\pi^{2n}}\sum_{\substack{1 \leq a \leq q}} \left( \mathcal{B}_{\underline{n},q,\underline{t}}(a) \mathcal{B}_{\underline{m},q,\underline{t}}(q-a)\right) \\
             &+ O_{n,k}\left(\frac{2^{2n + 2}(\log q^{n-1})^2}{q^{1-\epsilon}}\right) + O \left( \frac{(\log q)^{2n}}{\sqrt{q}} \right).
        \end{align*}
        
        Our aim is to only have one main term,
        \begin{align*}
            \frac{1}{\pi^{2n}}\sum_{\substack{a \geq 1}} \left( \mathcal{B}_{\underline{n},\underline{t}}(a) \mathcal{B}_{\underline{m},\underline{t}}(a)\right),
        \end{align*}
        so we want to first bound the term
        \begin{align}\label{eq: sigma minus}
            \frac{1}{\pi^{2n}}\sum_{\substack{1 \leq a \leq q}} \left( \mathcal{B}_{\underline{n},q,\underline{t}}(a) \mathcal{B}_{\underline{m},q,\underline{t}}(q-a)\right),
        \end{align}
        and then extend the sum 
        \begin{align*}
               \sum_{\substack{1 \leq a \leq q}} \left( \mathcal{B}_{\underline{n},q,\underline{t}}(a) \mathcal{B}_{\underline{m},q,\underline{t}}(a)\right) 
        \end{align*}
        over all positive integers.
        
        To bound Expression \eqref{eq: sigma minus} we again use the bound of $\mathcal{B}$ from Section \ref{s: bounding the moments} to show
        \begin{align*}
            \sum_{\substack{1 \leq a < q}} \left( \mathcal{B}_{\underline{n},q,\underline{t}}(a) \mathcal{B}_{\underline{m},q,\underline{t}}(q-a)\right) &\leq 2^{2n} \sum_{1 \leq a < q} \frac{d_n(a)d_k(a)}{a} \frac{d_n(q-a)d_k(q-a)}{q-a} \\
            &\ll q^\epsilon \sum_{1 \leq a < q} \frac{1}{a(q-a)} \leq q^\epsilon\frac{2 \log q}{q} .
        \end{align*}
        As a result,
        \begin{align*}
            M_q(\underline{n},\underline{n}) &= \frac{1}{\pi^{2n}}\sum_{\substack{1 \leq a \leq q}} \left( \mathcal{B}_{\underline{n},q,\underline{t}}(a) \mathcal{B}_{\underline{m},q,\underline{t}}(a)\right) + O_{n,k}\left(\frac{\log q}{q^{1 - \epsilon}}\right) + O_{n,k}\left(\frac{(\log q^{n-1})^2}{q^{1-\epsilon}}\right) + O \left( \frac{(\log q)^{2n}}{\sqrt{q}} \right). 
        \end{align*}
        This can be simplified, as $\mathcal{B}_{\underline{n},q,\underline{t}}$ is equivalent to $\mathcal{B}_{\underline{n},\underline{t}}$ in this case, and we can combine the errors. Since $k, \underline{n},\underline{m}$ are all fixed, we omit the dependencies on the error for ease of notation. Therefore
        \begin{align}\label{eq: moment with truncated main term}
            M_q(\underline{n},\underline{n}) &= \frac{1}{\pi^{2n}}\sum_{\substack{1 \leq a \leq q}} \left( \mathcal{B}_{\underline{n},\underline{t}}(a) \mathcal{B}_{\underline{m},\underline{t}}(a)\right) + O \left( \frac{(\log q)^{2n}}{\sqrt{q}} \right). 
        \end{align}
        
        The final step is to extend the main sum to infinity. We rewrite the sum
        \begin{align*}
            \sum_{\substack{1 \leq a \leq q}} \left( \mathcal{B}_{\underline{n},\underline{t}}(a) \mathcal{B}_{\underline{m},\underline{t}}(a)\right) = \sum_{\substack{a \geq 1}} \left( \mathcal{B}_{\underline{n},\underline{t}}(a) \mathcal{B}_{\underline{m},\underline{t}}(a)\right) - \sum_{\substack{a > q}} \left( \mathcal{B}_{\underline{n},\underline{t}}(a) \mathcal{B}_{\underline{m},\underline{t}}(a)\right).
        \end{align*}
        By bounding $\mathcal{B}$ as before, the second sum has the upper bound
        \begin{align*}
            \sum_{\substack{a > q}} \left( \mathcal{B}_{\underline{n},\underline{t}}(a) \mathcal{B}_{\underline{m},\underline{t}}(a)\right) \leq 2^{2n} \sum_{\substack{a > q}} \left( \frac{d^2_n(a)d^2_k(a)}{a^2}\right).
        \end{align*}
        We take $d_n(a)^2d_k(a)^2 = O(a^{2\epsilon_{k,n}}) =: O_k(a^\epsilon)$, so
        \begin{align}\label{eq: bounding main sum}
          \sum_{\substack{a > q}} \left( \mathcal{B}_{\underline{n},\underline{t}}(a) \mathcal{B}_{\underline{m},\underline{t}}(a)\right) \ll  \sum_{\substack{a > q}}  \frac{a^\epsilon}{a^2} \ll q^{-1+\epsilon}.
        \end{align}
    This bound is clearly smaller than the error term in Equation \eqref{eq: moment with truncated main term}, so as a result,
    \begin{align*}
        M_q(\underline{n},\underline{n}) &= \frac{1}{\pi^{2n}}\sum_{\substack{a \geq 1}} \left( \mathcal{B}_{\underline{n},\underline{t}}(a) \mathcal{B}_{\underline{m},\underline{t}}(a)\right) + O \left( \frac{(\log q)^{2n}}{\sqrt{q}} \right). 
    \end{align*}
    \end{proof}
    
    To finish proving Proposition \ref{thm: moment convergence} we prove Proposition \ref{thm: prop expectation to sum}, showing how the expectation also equals the sum
    \begin{align*}
        \frac{1}{\pi^{2n}}\sum_{\substack{a \geq 1}} \left( \mathcal{B}_{\underline{n},\underline{t}}(a) \mathcal{B}_{\underline{m},\underline{t}}(a)\right).
    \end{align*}
    
    \begin{proof}[Proof of Proposition \ref{thm: prop expectation to sum}]
    
    We are interested in the expectation
    \begin{align*}
       M(\underline{n},\underline{m}) =  \mathbb{E}\left( \prod_{i=1}^k F(t_i)^{n_i} \overline{F(t_i)}^{m_i}\right).
    \end{align*}
    Using Equation \eqref{eq: limiting moment} from Section \ref{s: defn of moments}, and
        $n\coloneqq n_1 + \cdots n_k = m_1 + \cdots + m_k$, the moment is equivalent to
    \begin{align*}
        M(\underline{n},\underline{m}) = \mathbb{E}\left( \frac{\eta^{n}\overline{\eta}^n}{\pi^{2n}}\right) \sum_{a \geq 1} \mathcal{B}_{\underline{n},\underline{t}}(a) \mathcal{B}_{\underline{m},\underline{t}}(a).
    \end{align*}

   Therefore we have
   \begin{align*}
       \mathbb{E}\left( \prod_{i=1}^k F_-(t_i)^{n_i} \overline{F_-(t_i)}^{m_i}\right) = \frac{1}{\pi^{2n}}  \sum_{\substack{a \geq 1}}\mathcal{B}_{\underline{n},\underline{t}}(a) \mathcal{B}_{\underline{m},\underline{t}}(a),
   \end{align*}
   proving Proposition \ref{thm: prop expectation to sum}.
   \end{proof}
    
    Therefore we have shown the multivariate moment sequence 
    \begin{align*}
        M_q(\underline{n},\underline{m}) = \frac{2}{\phi(q)} \sum_{\chi \odd} \prod_{i=1}^k f_\chi(t_i)^{n_i} \overline{f_\chi(t_i)}^{m_i}
    \end{align*}
    converges, as $q \rightarrow \infty$ through the primes, to 
    \begin{align*}
        \mathbb{E}\left( \prod_{i=1}^k F(t_i)^{n_i} \overline{F(t_i)}^{m_i} \right),
    \end{align*}
    for all $k$-tuples $\underline{n},\underline{m}$ and $0 \leq t_1 < \cdots < t_k \leq 1$. This section only addressed the odd character case, but the proof is similar for even characters and leads to the same results. Therefore $\left(\mathcal{F}_{q,\pm}\right)_{q \text{ prime}}$, the distribution of odd/even character paths $f_\chi$ modulo $q$, converges to $F_\pm$ as $q \rightarrow \infty$ in the sense of convergence in finite distributions.

\section{Relative Compactness of the Sequence of Distributions}\label{s: tightness}
    
    In the previous section we showed $(\mathcal{F}_q)$ converges in finite distributions to the process $F$ as $q \rightarrow \infty$ through the primes. If we can prove the sequence of distributions is relatively compact, then it follows that $(\mathcal{F}_q)$ converges in distribution to $F$ \cite[Example 5.1]{billingsley2013}. This is much stronger than convergence of finite-dimensional distributions and concludes the proof of Theorem \ref{thm: main theorem}.

    Prohorov's Theorem \cite[Theorem 5.1]{billingsley2013} states that if a sequence of probability measures is tight, then it must be relatively compact. For this we use Kolmorogorov's tightness criterion, quoted from Revuz and Yor:
    
    \begin{proposition}\cite[Th. XIII.1.8]{revuzyor2013}\label{thm: kolmorogorov tightness criterion} Let $(L_p(t))_{t \in [0,1]}$ be a sequence of $C([0,1])-$valued processes such that $L_p(0)=0$ for all $p$. If there exists constants $\alpha > 0$, $\delta > 0$ and $C \geq 0$ such that for any $p$ and any $s < t$ in $[0,1]$, and we have
    \begin{align*}
        \mathbb{E}(|L_p(t) - L_p(s)|^\alpha) \leq C|t-s|^{1 + \delta},
    \end{align*}
    then the sequence $(L_p(t))$ is tight.
    \end{proposition}

    For our sequence of processes $(\mathcal{F}_q(t))_{t \in [0,1]}$ we have $f_\chi(0) = 0$ for all $q$. We also have the trivial bound
    \begin{align*}
        |f_\chi(t) - f_\chi(s)| \leq \sqrt{q}|t-s|,
    \end{align*}
    leading to 
    \begin{align*}
        \mathbb{E}\big| f_\chi(t) - f_\chi(s) \big|^{2k} \leq q^{k} |t-s|^{2k}.
    \end{align*}
    As a result, for $k > 1$ if we take $|t-s| < \frac{1}{q^{1-\epsilon}}$ for $\epsilon \in (0, \frac{k-1}{2k - 1})$, we have
    \begin{align}\label{eq: tightness for small (t-s)}
        \mathbb{E}\big| f_\chi(t) - f_\chi(s) \big|^{2k} \leq |t-s|^{2k - \frac{k}{1 - \epsilon}} =: |t-s|^{1 + \delta_1},
    \end{align}
    where $\delta_1 \coloneqq \frac{k - 1 + \epsilon(1 - 2k)}{1 - \epsilon}$. Therefore if we show a similar bound for $\mathbb{E}\big| f_\chi(t) - f_\chi(s) \big|^{2k}$ for $|t -s| > \frac{1}{q^{1 - \epsilon}}$ then the tightness condition holds for our sequence of processes. 
    
   Various authors have found results bounding the average of the difference of character sums. For example,
    Cochrane and Zheng \cite{cochranezheng1998} prove for positive integers $k$ and Dirichlet characters modulo prime $q$,
        \begin{align*}
            \frac{1}{q - 1} \sum_{\chi \neq \chi_0} \left| \sum_{n = s + 1}^{s + t} \chi(n) \right|^{2k} \ll_{\epsilon,k} q^{k - 1 + \epsilon} + |t-s|^k q^\epsilon.
        \end{align*}
    To prove tightness however we need the $|t - s|$ term independent of $q$. 
    
    \begin{lemma}\label{thm: tightness for large (t-s)}
    Let $q$ be an odd prime. For all $\epsilon \in (0,1)$, there exists absolute constants $C_1(\epsilon), C_2$ independent of $q$ such that for all $0 \leq s < t \leq 1$,
    \begin{align*}
        \mathbb{E}\big| f_\chi(t) - f_\chi(s) \big|^{4} \leq C_1(\epsilon) |t-s|^{1 + \delta_2} + C_2 \frac{(\log q)^4}{q},
    \end{align*}
    where $\delta_2 \coloneqq 1 - \epsilon$.
    \end{lemma}
    
    This lemma can be applied to characters of all moduli, not just primes, but for our work it is sufficient to look only at primitive characters. Clearly if $|t-s| \geq \frac{(\log q)^4}{q}$ then the equation becomes 
    \begin{align*}
        \mathbb{E}\big| f_\chi(t) - f_\chi(s) \big|^{4} \leq C |t-s|^{1 + \delta},
    \end{align*}
    which, combined with Equation \eqref{eq: tightness for small (t-s)} above, proves the sequence $(\mathcal{F}_q)$ is tight for all $s,t \in [0,1]$.
    
    Lemma \ref{thm: tightness for large (t-s)} is similar to a result of Bober and Goldmakher \cite[Lemma 4.1]{bobergoldmakher2013} and we use parts of their work in the proof. Unlike Section \ref{s: method of moments}, we will consider the odd and even case at the same time.
    
    \begin{proof}[Proof of Lemma \ref{thm: tightness for large (t-s)}]
        Using the Fourier expansion of $f_\chi$, the difference $(f_\chi(t) - f_\chi(s))$ can be written as
        \begin{align*}
            \frac{\tau(\chi)}{2 \pi i \sqrt{q}} \sum_{1 \leq |n| \leq q} \frac{\overline{\chi}(n)}{n}e(-sn) \left( 1 - e(-(t-s)n)\right) + O\left( \frac{\log q}{\sqrt{q}} \right).
        \end{align*}
        Consequently,
        \begin{align*}
            \big| f_\chi(t) - f_\chi(s) \big|^4 \leq \frac{2^4}{\pi^4} \bigg| \sum_{1 \leq n \leq q} \frac{\overline{\chi}(n)}{n}e(-sn) \left( 1 - e(-(t-s)n)\right) \bigg|^4 + O\left( \frac{(3 + \log q)^4}{q^2}\right).
        \end{align*}
        Similar to Section \ref{s: defn of moments} and \cite[Lemma 4.1]{bobergoldmakher2013}, we define
        \begin{align*}
            b(n) = \left\{ \begin{array}{c|l}\frac{1}{n}\sum_{\substack{n_1 n_2 = n \\ n_i \leq q}} \prod_{j=1}^2 \left( e(-sn_j)\left( 1 - e\left(-(t-s)n_j\right) \right)\right) & (n,q) = 1, \\
            0 & \text{otherwise}.\end{array}\right.
        \end{align*}
        Therefore,
        \begin{align*}
            \bigg|\sum_{1 \leq n \leq q} \frac{\overline{\chi}(n)}{n}e(-sn) \left( 1 - e(-(t-s)n)\right) \bigg|^4 = \bigg| \sum_{1 \leq n \leq q^2} \overline{\chi}(n) b(n)\bigg|^2.
        \end{align*}
        
        The sum $b(n)$ can be bounded using $(1 - e(x)) \leq \min \{ 2, x\}$, so
        \begin{align}\label{eq: bound of b(n)}
            |b(n)| \leq d(n) \min\bigg\{ \frac{2^2}{n}, \left(2\pi(t - s)\right)^2 \bigg\}.
        \end{align}
        As a result, taking $n = a + mq$,
        \begin{align}\label{eq: splitting b(n) and bounding by the divisor function}
            \mathbb{E} \bigg| \sum_{1 \leq n \leq q^2} \overline{\chi}(n) b(n)\bigg|^2 = \sum_{a = 1}^q\bigg| \sum_{m = 0}^{q} b(a + mq)\bigg|^2
            \leq 4\sum_{a = 1}^q |b(a)|^2 + 4 \sum_{a = 1}^q\bigg| 2^2 \sum_{m=1}^{q} \frac{ d(a + mq)}{a + mq}\bigg|^2.
        \end{align}
   
    We are interested in bounding the latter inner sum,
        \begin{align*}
            \mathcal{D}_a \coloneqq \sum_{m=1}^{q} \frac{ d(a + mq)}{a + mq} = \sum_{\substack{q < m\leq (a + q^2) \\ m \equiv a (q)}} \frac{ d(m)}{m}.
        \end{align*}
    By Abel summation this is
    \begin{align*}
        \frac{1}{a + q^2} \sum_{\substack{ q < m \leq (a + q^2) \\ m \equiv a (q)}} d(m) + \int_q^{(a + q^2)} \frac{1}{t^2} \sum_{\substack{m \leq t \\ m \equiv a (q)}} d(a) dt. 
    \end{align*}
    In order to further bound the sum, we use the Shiu's upper bound \cite[Theorem 1]{shiu},
    \begin{align*}
        \sum_{\substack{n \leq x \\ n \equiv a \, (q)}} d(n) \ll_\delta \frac{x \phi(q)\log x}{q^2} < x \cdot \frac{\log x}{q},
    \end{align*}
    which is valid for all $x \geq q^{1 + \delta}$ for any $\delta > 0$. Therefore,
    \begin{align*}
        \mathcal{D}_a &= \frac{1}{a + q^2} \sum_{\substack{q \leq  m \leq (a + q^2) \\ m \equiv a (q)}} d(m) + \int_{q}^{(a + q^2)} \frac{1}{t^2} \sum_{\substack{m \leq t \\ m \equiv a (q)}} d(a) dt \\
        &=O\left( \frac{\log(a + q^2)}{q} \right) + O\left(\int_{q^{1 + \delta}}^{(a + q^2)} \frac{\log t}{q t}  dt\right) + \int_{q}^{q^{1+ \delta}} \frac{1}{t^2} \sum_{\substack{m \leq t \\ m \equiv a (q)}} d(a) dt. 
    \end{align*}
    
    The latter integral is bounded by the following results \cite[Equation 27.11.2]{NIST:DLMF}\cite{fouvry1992}:
    \begin{align*}
        \sum_{m \leq t} d(m) &= t \log t + (2\gamma - 1)t + O(\sqrt{t}) = O(t\log t), \\
        \sum_{\substack{m \leq t \\ m \equiv a (q)}} d(a) &= \frac{1}{\phi(q)} \sum_{\substack{m \leq t \\ (m,q) = 1}} d(m) + O\left((q^{1/2} + t^{1/3})t^{\epsilon}\right) = O\left( \frac{t \log t}{q} \right).
    \end{align*}
    Combining these equations, we have
    \begin{align*}
        \mathcal{D}_a = O\left( \frac{\log q}{q} \right) + O_\delta\left(\frac{(\log q)^2}{q} \right) + O_\delta\left( \frac{(\log q)^2}{q}\right).
    \end{align*}
    
    As a result, fixing $\delta > 0$, Equation \eqref{eq: splitting b(n) and bounding by the divisor function} becomes
    \begin{align*}
        \sum_{a = 1}^q\bigg| \sum_{m = 0}^{q} b(a + mq)\bigg|^2
            \leq 4\sum_{a = 1}^q |b(a)|^2 + O\left( \sum_{a = 1}^q \left| \frac{(\log q)^2}{q}\right|^2 \right)
            = 4\sum_{a = 1}^q |b(a)|^2 + O\left( \frac{(\log q)^4}{q} \right).
    \end{align*}
        Therefore the only sum left to evaluate is $\sum_{a \leq q} |b(a)|^2$. Using the bound from Equation \eqref{eq: bound of b(n)} and splitting the cases $\frac{1}{a} > \pi^2(t - s)^2$ and $\frac{1}{a} < \pi^2 (t-s)^2$, we have
        \begin{align*}
            \sum_{a = 1}^q |b(a)|^2 \leq 2^4 \left( \pi^4 (t-s)^4\sum_{a \leq \pi^{-2} (t-s)^{-2}} d(a)^2 + \sum_{\pi^{-2}(t-s)^{-2}< a \leq q} \frac{d(a)^2}{a^2} \right).
        \end{align*}
        We combine the two sums by Rankin's trick. Taking $x = \pi^{-2}(t-s)^{-2}$, 
        \begin{align*}
            \frac{1}{x^2}\sum_{a \leq x} d(a)^2 &\leq \frac{x^{\sigma_1}}{x^2} \sum_{a = 1}^\infty \frac{d(a)^2}{a^{\sigma_1}}, &1 < \sigma_1 < 2, \\
            \sum_{a \geq x} \frac{d(a)^2}{a^2} &\leq \frac{1}{x^{\sigma_2}} \sum_{a = 1}^\infty \frac{d(a)^2}{a^{2 - \sigma_2}}, &0 < \sigma_2 < 1.
        \end{align*}
        Note $\sigma_1,\sigma_2$ are bounded so that the sums converge and tend to zero as $x \rightarrow \infty$.
        These sums are addressed by one of Ramanujan's identities \cite[Section 1.3.1, Question 5]{montgomeryvaughan2006}. For $Re(s) > 1$, $$\sum_{n =1}^\infty \frac{d(n)^2}{n^{s}} = \frac{\zeta(s)^4}{\zeta(2s)}.$$  
        Therefore
        \begin{align*}
             \sum_{a = 1}^q |b(a)|^2 \leq 2^4 \left(\frac{1}{x^{2 - \sigma_1}} \frac{\zeta(\sigma_1)^4}{\zeta(2\sigma_1)} + \frac{1}{x^{\sigma_2}}\frac{\zeta(2 - \sigma_2)^4}{\zeta(2(2 - \sigma_2))}\right).
        \end{align*}
        Taking $\sigma \coloneqq \min(2 - \sigma_1, \sigma_2) \in (0,1)$ and substituting back $\pi^{-2} (t-s)^{-2} = x$,
        \begin{align*}
            \frac{1}{x^{2 - \sigma_1}} \frac{\zeta(\sigma_1)^4}{\zeta(2\sigma_1)} + \frac{1}{x^{\sigma_2}}\frac{\zeta(2 - \sigma_2)^4}{\zeta(2(2 - \sigma_2))}
            \leq \frac{C}{x^{\sigma}} = C\pi^{2\sigma} (t-s)^{2\sigma},
        \end{align*}
        for some $C = C(\sigma) > 0$. As a result,
        \begin{align*}
            \mathbb{E}\big| f_\chi(t) - f_\chi(s)\big|^4 \leq C (t-s)^{2\sigma} + O\left( \frac{(\log q)^4}{q}\right).
        \end{align*}
       Taking $\sigma = 1 - \epsilon$ and therefore $2\sigma = 2 - \epsilon$, we have completed the proof. 
    \end{proof}
    
    Lemma \ref{thm: tightness for large (t-s)} shows that the Kolmorogorov's tightness criterion argument holds for 
    \begin{align*}
        |t-s|^{1 + \delta_2} \gg \frac{(\log q)^4}{q},
    \end{align*}
    where we take $\alpha = 4$ from Proposition \ref{thm: kolmorogorov tightness criterion}.
    Therefore, combining with Equation \eqref{eq: tightness for small (t-s)}, we have shown for constant $K$ 
    \begin{align*}
        \mathbb{E}\big| f_\chi(t) - f_\chi(s)\big|^4 \leq \left\{ \begin{array}{cl}
            |t-s|^{1 + \delta_1}, & |t-s| \leq q^{-(1 - \epsilon_1)} \\
            K |t-s|^{1 + \delta_2}, & |t-s| \geq (\log q)^{4}q^{-1}
        \end{array}\right.,
    \end{align*}
    where we can choose $\delta_1$ and $\delta_2$ in such a way that
    for $\delta_1 = \frac{1 - 3\epsilon_1}{1 - \epsilon_1}$ for $\epsilon_1 \in (0,\frac{1}{3})$ and and $\delta_2 = 1 - \epsilon_2$ where  $\epsilon_2 \in (0,1)$. This is possible as our initial parameter choices are flexible enough to allow this.
    
    For large enough $q$, we have 
    \begin{align*}
            \frac{(\log q)^4}{q}< \frac{1}{q^{1 - \epsilon_1}}.
    \end{align*}
   Therefore taking $\delta \coloneqq \min(\delta_1,\delta_2)$, Kolmorogorov's tightness criterion holds for all $t,s$ and $(\mathcal{F}_q)$ is tight. As a result, $(\mathcal{F}_q)_{q \text{ prime}}$ converges in distribution to the random process $F$ as $q \rightarrow \infty$, proving Theorem \ref{thm: main theorem}.
    This concurs with the result from Bober, Goldmakher, Granville and Koukoulopoulos for their distribution function 
    \begin{align*}
    \Phi_q(\tau) \coloneqq \frac{1}{\phi(q)} \# \bigg \{ \chi \mod q : \max_t |S_\chi(t)| > \frac{e^\gamma}{\pi}\tau \bigg \}
    \end{align*}
    weakly converging to their limiting function \cite[Theorem 1.4]{bggk2014}
    \begin{align*}
        \Phi(\tau) \coloneqq \prob\left( \max_t |F(t)| > 2 e^\gamma \tau \right).
    \end{align*}

    \bibliography{bibliography1}
    \bibliographystyle{plain}
    
    \Addresses
\end{document}